\documentclass{amsart}
\usepackage{amssymb,amsthm, amsmath, amsfonts, amsbsy}
\usepackage{graphics}
\usepackage{hyperref}
\usepackage[all]{xy}
\usepackage{mathabx}
\usepackage{enumerate}
\usepackage[mathscr]{eucal}
\usepackage{bbm}

\theoremstyle{plain}
\newtheorem{theorem}{Theorem}[section]
\newtheorem{lemma}[theorem]{Lemma}
\newtheorem{proposition}[theorem]{Proposition}

\newtheorem{corollary}[theorem]{Corollary}
\numberwithin{equation}{section}

\theoremstyle{definition}

\newtheorem{definition}[theorem]{Definition}
\newtheorem{example}[theorem]{Example}
\newtheorem{remark}[theorem]{Remark}

\DeclareMathOperator{\Mod}{-Mod}

\DeclareMathOperator{\Hom}{Hom}
\DeclareMathOperator{\Ext}{Ext}
\DeclareMathOperator{\Tor}{Tor}
\DeclareMathOperator{\hd}{hd}
\DeclareMathOperator{\gd}{gd}
\DeclareMathOperator{\pd}{pd}
\DeclareMathOperator{\op}{op}

\DeclareMathOperator{\supp}{supp}
\DeclareMathOperator{\tht}{ht}
\DeclareMathOperator{\ts}{ts}

\DeclareMathOperator{\Ob}{Ob}
\DeclareMathOperator{\tor}{tor}
\newcommand{\Sh}{{\mathrm{Sh}}}

\newcommand{\BS}{\boldsymbol{\Sigma}}
\newcommand{\C}{{\mathscr{C}}}
\newcommand{\D}{{\mathscr{D}}}
\newcommand{\N}{{\mathscr{N}}}
\newcommand{\Y}{{\mathscr{Y}}}

\newcommand{\bfm}{{\mathbf{m}}}
\newcommand{\bfn}{{\mathbf{n}}}
\newcommand{\bfx}{{\mathbf{x}}}
\newcommand{\bfy}{{\mathbf{y}}}
\newcommand{\bfz}{{\mathbf{z}}}
\newcommand{\bfw}{{\mathbf{w}}}

\newcommand{\bfo}{{\mathbf{1}}}

\newcommand{\FI}{{\mathrm{FI}}}

\title{Representations of $\mathbb{N}^{\infty}$-type combinatorial categories}

\author{Zhenxing Di}
\address{School of Mathematical Sciences, Huaqiao University, Quanzhou 362021, China}
\email{dizhenxing@163.com}

\author{Liping Li}
\address{LCSM(Ministry of Education), School of Mathematics and Statistics, Hunan Normal University, Changsha 410081, China}
\email{lipingli@hunnu.edu.cn}

\author{Li Liang}
\address{Department of Mathematics, Lanzhou Jiaotong University, Lanzhou 730070, China}
\email{lliangnju@gmail.com}

\thanks{Z. Di was partly supported by NSF of China (Grant No. 12471034), the Scientific Research Funds of Huaqiao University (Grant No. 605-50Y22050) and the Fujian Alliance Of Mathematics (Grant No. 2024SXLMMS04); L. Li was partly supported by NSF of China (Grant No. 12171146); L. Liang was partly supported by NSF of China (Grant No. 12271230).}

\thanks{The authors greatly appreciate the anonymous reviewer for carefully checking the paper and providing many valuable comments and suggestions.}

\begin{document}

\begin{abstract}
In this paper we consider representations of certain combinatorial categories, including the poset $\D$ of positive integers and division, the Young lattice $\mathscr{Y}$ of partitions of finite sets, the opposite category of the orbit category $\mathscr{Z}$ of $(\mathbb{Z}, +)$ with respect to nontrivial subgroups, and the category $\mathscr{CI}$ of finite cyclic groups and injective homomorphisms. We describe explicit upper bounds for homological degrees of their representations, and deduce that finitely presented representations (resp., representations presented in finite degrees) over a field form abelian subcategories of the representation categories. We also give an explicit description for the category of sheaves over the ringed atomic site $(\mathscr{Z}, \, J_{at}, \, \underline{\mathbb{C}})$, and show that irreducible sheaves are parameterized by primitive roots of the unit.
\end{abstract}

\maketitle

\section{Introduction}

The poset $\D$ of positive integers and division is a well known example of lattices, and is closely related to quite a few important posets in combinatorics. For example, by the fundamental theorem of arithmetic, it is isomorphic to the poset $\mathbb{N}^{\infty}$ whose elements are sequence $\bfx = (x_1, x_2, \ldots)$ of natural numbers with all but finitely many entries being 0, and whose order is given by $\bfx \leqslant \bfy$ if $x_i \leqslant y_i$ for each $i$. Furthermore, the infinite symmetric group $S_{\infty} = \varinjlim_{n} S_n$ acts on $\mathbb{N}^{\infty}$ in a natural way and this action respects the partial order, so the quotient set consisting of orbits is again a poset, and is isomorphic to the Young latttice $\mathscr{Y}$ whose elements are partitions of $[n] = \{1, 2, \ldots, n\}$ with $n$ ranging over all positive integers and whose partial order is given by containment.

These posets and related combinatorial categories have been extensively considered in various areas such as combinatorics, combinatorial topology, operad theory and representation theory of symmetric groups or Lie algebras \cite{Dh, Du, Sta, St, Su, V, Zig}. Recently, their representations also attract attentions of experts in many fields. For instances, \textit{persistence modules} are widely applied in topological data analysis and related areas \cite{BL, BS, Ou}, and $\FI$-modules play a central role in representation stability theory \cite{CEFN}. Thus representation theory of posets and combinatorial categories is a valuable topic, not only because of its own interest, but also because of the rich application in other areas such as algebraic topology, commutative algebra, applied combinatorics and applied topology. As a particular example, we note that representations of $\D$ are precisely multi-graded modules of the polynomial ring with countably many variables, and correspondingly representations of $\mathscr{Y}$ are closely related to $S_{\infty}$-symmetric modules of this polynomial ring, which are actively studied by quite a few authors; see \cite{AH, Dr, LNNR, NR, NR1} and in particular the nice short survey paper \cite{JLR}.

In this paper we study representations of categories $\C$ sharing the following combinatorial features of $\D$, called \textit{combinatorial categories of $\mathbb{N}^{\infty}$-type}:
\begin{itemize}
\item $\C$ is a small skeletal \textit{EI category}, that is, every endomorphism in $\C$ is an isomorphism;

\item there is a functor $\iota: \D \to \C$ which is a bijection between the object sets;

\item for each pair of objects $x$ and $y$, the group $G_x = \C(x, x)$ acts freely and transitively on $\C(x, y)$ from the right side, while the group $G_y$ acts transitively on $\C(x, y)$ from the left side.
\end{itemize}
Besides the poset $\D$, typical examples include:
\begin{enumerate}
\item the opposite category of the orbit category $\mathscr{Z}$ of $(\mathbb{Z}, +)$ with respect to the family of nontrivial subgroups;

\item the opposite category of the category $\mathscr{CS}$ of finite cyclic groups and surjective group homomorphisms;

\item the category $\mathscr{CI}$ of finite cyclic groups and injective group homomorphisms.
\end{enumerate}

Let $\C$ be a combinatorial category of $\mathbb{N}^{\infty}$-type, and $k$ a commutative ring. A \textit{representation} of $\C$ (or called a \textit{$\C$-module}) is a covariant functor from $\C$ to $k \Mod$, the category of all $k$-modules. The category $\C \Mod$ is a rather big abelian category, in practice one may work in relatively smaller abelian subcategories consisting of $\C$-modules satisfying certain extra conditions. However, the usual notion of finitely generated modules is not a proper choice. Therefore, we want to find alternative conditions such that $\C$-modules satisfying them form an abelian subcategory of $\C \Mod$. The first result of this aspect is:

\begin{theorem} \label{main result 1}
Let $\C$ be a combinatorial category of $\mathbb{N}^{\infty}$-type and $k$ a field. If the endomorphism group $G_{\bfx}$ is finite for each $\bfx \in \Ob(\C)$, then the category of finitely presented $\C$-modules is abelian.
\end{theorem}

There are many interesting finitely generated $\C$-modules (even the simple modules) are not finitely presented. Moreover, for practical purpose people frequently have to deal with infinitely generated $\C$-modules. In this case we introduce a few homological numerical invariants and impose some finiteness conditions on them. It turns out that $\C$-modules satisfying these conditions also form an abelian subcategory of $\C \Mod$.

Let us give some details about these finiteness conditions. Since the object set $\Ob(\C)$ can be identified with the set of finite sequences $\bfx = (x_1, \ldots, x_n)$ of natural numbers, we impose the \textit{sup norm} on $\Ob(\C)$ by setting $\| \bfx \|$ to be the maximal entry in the sequence. Given a $\C$-module $V$, we determine the smallest subset $S \subseteq \Ob(\C)$ such that $V$ is generated by its values on objects in $S$. The zeroth \textit{homological degree} $\hd_0(V)$ is defined to be the supremum of those $\| \bfx \|$ for $\bfx \in S$. By a standard homological process we can define $\hd_i(V)$ for all $i \geqslant 0$. We say that $V$ is \textit{presented in finite degrees} if both $\hd_0(V)$ and $\hd_1(V)$ are finite.

\begin{theorem} \label{main result 2}
Let $\C$ be a combinatorial category of type $\mathbb{N}^{\infty}$, $k$ a commutative ring, $\| \bullet \|$ the sup norm on $\Ob(\C)$, and $V$ a $\C$-module. Then
\[
\hd_s(V) \leqslant \max \{\hd_0(V), \, \hd_1(V) \}
\]
for $s \geqslant 0$. In particular, the category of $\C$-modules presented in finite degrees is abelian.
\end{theorem}

At a first glance this result seems unexpectedly simple. This happens because we define $\| \bfx \|$ to be the maximal entry in $\bfx$, which is very weak in the sense that too many distinct objects share the same norm. Instead, if we define homological degrees with respect to strong norms, for instance, the \textit{sum norm} by setting $\| \bfx \|$ to be the sum of all entries in $\bfx$, then we cannot bound higher homological degrees in terms of the first two; see Example \ref{counterexample}. But for the Young lattice (which as a quotient category of $\mathbb{N}^{\infty}$ is not a combinatorial category of type $\mathbb{N}^{\infty}$), we do get an affirmative answer.

\begin{theorem} \label{main result 3}
Let $\mathscr{Y}$ be the Young lattice, $k$ a field, and $\| \bullet \|$ the sum norm on $\Ob(\mathscr{Y})$. If $s \geqslant 2$ and $\hd_{s-1}(V) > 0$, then
\[
\hd_s(V) \leqslant \hd_{s-1}(V) (\ln(\hd_{s-1}(V)) + 1).
\]
In particular, the category of $\mathscr{Y}$-modules presented in finite degrees is abelian.
\end{theorem}

Another important subcategory of $\C \Mod$ which has a natural abelian structure (but not inherited from that of $\C \Mod$) is the category $\Sh(\C^{\op}, \, J, \, \underline{k})$ of sheaves over $\C^{\op}$ equipped with a certain Grothendieck topology $J$ and the constant structure sheaf $\underline{k}$. The main motivation for us to consider this category lies in three aspects: the sheaf category over orbit categories equipped with the atomic Grothendieck topology $J_{at}$ is closely related to continuous representations of topological groups (see for instances \cite{DLLX, Jo, MM}); the sheaf theoretic approach can provide a novel viewpoint for representations, for instance, the complicated shift functor constructed in Section 3 turns out to be the sheafification functor; and finally, for some examples studied in this paper, the sheaf category is much simpler than the representation category. For instance, if $\C$ is a poset, then $\Sh(\C^{\op}, \, J_{at}, \, \underline{k})$ is equivalent to $k \Mod$. Although the situation becomes more complicated for other examples considered in this paper, we still obtain the following result for a specific category.

\begin{theorem} \label{main result 4}
Let $k$ be an algebraically closed field with characteristic 0, $\mathscr{Z}$ the orbit category of $(\mathbb{Z}, +)$ with respect to the family of nontrivial subgroups, and $J_{at}$ the atomic topology on $\mathscr{Z}$. Then $\Sh(\mathscr{Z}, \, J_{at}, \, \underline{k})$ is a semisimple category, and irreducible sheaves are parameterized by primitive roots of the unit in $k$.
\end{theorem}

At this moment we have not obtained similar results for $\Sh(\mathscr{CS}, \, J_{at}, \, \underline{k})$ or $\Sh(\mathscr{CI}^{\op}, \, J_{at}, \, \underline{k})$. The main obstacle is that an explicit description for automorphisms groups of finite cyclic groups is not available.

This paper is organized as follows: in the second section we introduce categories of type $\mathbb{N}^{\infty}$ and some homological invariants of their representations. Section 3 is devoted to representations of $\mathbb{N}^{\infty}$. We give an upper bound for homological degrees, and investigate shift functors in details. In Section 4 we prove Theorems \ref{main result 1} and \ref{main result 2} for combinatorial categories $\C$ of type $\mathbb{N}^{\infty}$, relying on the assumption that $\mathbb{N}^{\infty}$ can be viewed as a subcategory of $\C$ and the observation that the restriction functor $\C \Mod \to \mathbb{N}^{\infty} \Mod$ has a nice property. Representations of Young lattice is considered in Section 5, where we prove Theorem \ref{main result 3}. A sheaf theoretic approach as well as a proof of Theorem \ref{main result 4} is given in the last section.

\section{Preliminaries}

In this section we introduce combinatorial categories of type $\mathbb{N}^{\infty}$, describe some examples, and propose a few important invariants for their representations.

\subsection{Combinatorial categories of type $\mathbb{N}^{\infty}$}

Let $\C$ be a small skeletal EI category; that is, every endomorphism in $\C$ is an isomorphism. Then $G_x = \C(x, x)$ is a group for every object $x$ in $\C$. Furthermore, the binary relation $\preccurlyeq$ on $\Ob(\C)$ given by $ x \preccurlyeq y$ if $\C(x, y) \neq \emptyset$ is a partial order.

\begin{definition}
A small skeletal EI category $\C$ is said to be a combinatorial category of type $\mathbb{N}^{\infty}$ if the following conditions hold:
\begin{enumerate}
\item there is a functor $\iota: \mathbb{N}^{\infty} \to \C$ which is a bijection between the object sets;

\item the group $G_x$ acts regularly (i.e., freely and transitively) on $\C(x, y)$ from the right side for $x, y \in \Ob(\C)$;

\item the group $G_y$ acts transitively on $\C(x, y)$ from the left side for $x, y \in \Ob(\C)$
\end{enumerate}
\end{definition}

By the first condition, we can identify the underlying poset $(\Ob(\C), \preccurlyeq)$ of $\C$ and $\mathbb{N}^{\infty}$, so in the rest of this paper by abuse of notation we frequently denote objects in $\C$ by $\bfx$ with $\bfx \in \mathbb{N}^{\infty}$.

Recall the following definition of left Ore condition: given morphisms $f: \bfx \to \bfy$ and $g: \bfx \to \bfz$ in $\C$, there exists an object $\bfw$ as well as morphisms $f': \bfy \to \bfw$ and $g': \bfz \to \bfw$ such that the diagram commutes
\[
\xymatrix{
\bfx \ar[r]^-f \ar[d]^-g & \bfy \ar[d]^-{f'} \\
\bfz \ar[r]^-{g'} & \bfw.
}
\]
Dually, one can define right Ore condition.

The following facts can be easily deduced from definitions.

\begin{lemma} \label{facts}
Let $\C$ be a combinatorial category of type $\mathbb{N}^{\infty}$. Then one has:
\begin{enumerate}
\item every morphism in $\C$ is a monomorphism;

\item $\C$ satisfies the left Ore condition.
\end{enumerate}
\end{lemma}

\begin{proof}
(1) Given two morphisms $\alpha$ and $\alpha'$ from $\bfx$ to $\bfy$ and a morphism $\beta$ from $\bfy$ to $\bfz$ such that $\beta \alpha = \beta \alpha'$, since $G_{\bfx}$ acts regularly on $\C(\bfx, \bfy)$, we can find an element $g \in G_{\bfx}$ such that $\alpha' = \alpha g$. Consequently, $\beta \alpha = \beta \alpha g$. But $G_{\bfx}$ also acts regularly on $\C(\bfx, \bfz)$, this happens if and only if $g$ is the identity in $G_{\bfx}$, so $\alpha = \alpha'$.

(2) Since $\C$ is of type $\mathbb{N}^{\infty}$, we can let $\bfw$ be a common upper bound of $\bfx$ and $\bfz$ and take $f': \bfy \to \bfw$ and $g_1: \bfz \to \bfw$ to be arbitrary morphisms. Since $G_{\bfw}$ acts transitively on $\C(\bfx, \, \bfw)$, we can find a certain $\delta \in G_{\bfw}$ such that $\delta \circ g_1 \circ g = f' \circ f$. Then $\delta \circ g_1$ is the desired $g'$.
\end{proof}

Here are some interesting examples of combinatorial categories of type $\mathbb{N}^{\infty}$.

\begin{example}
Let $\D$ be the poset of positive integers and division. This poset is isomorphic to $\mathbb{N}^{\infty}$. Indeed, we can order all primes as follows $p_1 < p_2 < p_3 < \ldots$ and the fundamental theorem of arithmetic asserts that every positive integer $n$ uniquely corresponds a sequence $(x_1, \, x_2, \, \ldots)$ such that $n = p_1^{x_1} p_2^{x_2} \ldots$. It is clear that the assignment $n \mapsto (x_1, \, x_2, \, \ldots)$ gives the desired isomorphism of posets.
\end{example}

\begin{example} \label{orbit cat of Z}
Let $G = (\mathbb{Z}, +)$ and $\mathcal{H}$ the set of nontrivial subgroups of $G$. Define $\mathscr{Z}$ to be the orbit category of $G$ with respect to $\mathcal{H}$. Explicitly, objects are cosets $\bfn = \mathbb{Z}/n\mathbb{Z}$ with $n \in \mathbb{N}_+$, and morphisms $\bfm$ to $\bfn$ are $G$-equivariant maps. We check that $\mathscr{Z}^{\op}$ is a combinatorial category of type $\mathbb{N}^{\infty}$.
\begin{enumerate}
\item $\mathscr{Z}(\bfn, \, \bfn)$ is a cyclic group of order $n$, so $\mathscr{Z}$ is a small skeletal EI category, and so is $\mathscr{Z}^{\op}$.

\item There is an embedding functor $\iota: \D \to \mathscr{Z}^{\op}$ sending a morphism $m \to n$ in $\D$ to the following $G$-equivariant map
\[
\mathbb{Z}/n\mathbb{Z} \to \mathbb{Z}/m\mathbb{Z}, \quad 1 + n\mathbb{Z} \mapsto 1 + m \mathbb{Z}
\]
in $\mathscr{Z}$, which is a morphism in $\mathscr{Z}^{\op} (\bfm, \bfn)$. Clearly, this functor gives a bijection on object sets.

\item A direct computation shows that $\mathscr{Z} (\bfn, \bfn)$ (resp., $\mathscr{Z} (\bfm, \bfm)$) acts regularly (resp., transitively) on $\mathscr{Z} (\bfm, \, \bfn)$ from the left (resp., right) side when $n$ divides $m$, so $\mathscr{Z}^{\op} (\bfn, \bfn)$ (resp., $\mathscr{Z}^{\op} (\bfm, \bfm)$) acts regularly (resp., transitively) on $\mathscr{Z}^{\op} (\bfn, \, \bfm)$ from the right (resp., left) side.
\end{enumerate}
\end{example}

\begin{example} \label{cyclic groups and injections}
Let $\mathscr{CI}$ be the category whose objects are finite cyclic groups $\bfn = \mathbb{Z}/n\mathbb{Z}$ with $n \in \mathbb{N}_+$, and morphisms are injective group homomorphisms. We check that $\mathscr{CI}$ is a combinatorial category of type $\mathbb{N}^{\infty}$.
\begin{enumerate}
\item There is a natural embedding functor from $\D$ to $\mathscr{CI}$ sending a morphism $m \to n$ in $\D$ to the following group homomorphism
\[
\mathbb{Z}/m\mathbb{Z} \to \mathbb{Z}/n\mathbb{Z}, \quad 1 + m\mathbb{Z} \mapsto \frac{n}{m} + n\mathbb{Z}.
\]

\item When $n$ divides $m$, $\bfm$ contains a unique subgroup isomorphic to $\bfn$, so $\mathscr{CI}(\bfn, \, \bfn)$ acts regularly on $\mathscr{CI}(\bfn, \bfm)$ from the right side.

\item When $n$ divides $m$, every automorphism of $\bfn$ can extend to an automorphism of $\bfm$, so $\mathscr{CI}(\bfm, \, \bfm)$ acts transitively on $\mathscr{CI}(\bfn, \bfm)$ from the left side.
\end{enumerate}
\end{example}

\begin{example} \label{cyclic groups and surjections}
Let $\mathscr{CS}$ be the category whose objects are finite cyclic groups $\bfn = \mathbb{Z}/n\mathbb{Z}$ with $n \in \mathbb{N}_+$, and morphisms are surjective group homomorphisms. The reader can check that $\mathscr{CS}^{\op}$ is a combinatorial category of type $\mathbb{N}^{\infty}$. Note that $\mathscr{CS}$ is different from the orbit category $\mathscr{Z}$ in Example \ref{orbit cat of Z} since here we regard $\bfn$ as a $\mathbb{Z}$-module, while in Example \ref{orbit cat of Z} we only regard $\bfn$ as a $\mathbb{Z}$-set of the additive group $\mathbb{Z}$.
\end{example}

From now on let $\C$ be a combinatorial category of type $\mathbb{N}^{\infty}$. A function $\| \bullet \|: \Ob(\C) \to \mathbb{N}$ is called a \textit{norm} if it satisfies the following conditions:
\begin{enumerate}
\item $\| \bfx \| \geqslant 0$ and the equality holds if and only if $\bfx = \boldsymbol{0}$, the sequence whose each entry is 0;

\item $\| a \bfx \| = a \|\bfx \|$ for each $a \in \mathbb{N}$;

\item $\| \bfx + \bfy \| \leqslant \| \bfx \| + \|\bfy \|$.
\end{enumerate}
In this paper we mainly consider the following norms, called the \textit{sup norm} and the \textit{sum norm} respectively:
\[
\bfx \mapsto \sup \{x_i \mid i \in \mathbb{N}_+ \}; \quad \bfx \mapsto \sum_{i \in \mathbb{N}_+} x_i.
\]

\subsection{Representations and homological invariants}

Throughout this subsection let $k$ be a commutative ring. A representation of $\C$ over $k$, or a $\C$-module, is a covariant functor from $\C$ to $k \Mod$. Denote the value of $V$ on an object $\bfx$ by $V_{\bfx}$, and the \textit{support} of $V$ is set to be
\[
\supp(V) = \{\bfx \in \Ob(\C) \mid V_{\bfx} \neq 0 \}.
\]
The category of all $\C$-modules is a Grothendieck category admitting enough projectives. For every object $\bfx$ in $\C$, denote by $P(\bfx)$ the $k$-linearization of the representable functor $\C(\bfx, -)$.

Let $k\C$ be the category algebra, which is a free $k$-module with a basis consisting of morphisms in $\C$, and whose multiplication is defined by setting $f \cdot g = f \circ g$ if they can be composed, and otherwise setting $f \cdot g = 0$. It is well known that there is an embedding functor
\[
\iota: \C \Mod \to k\C \Mod, \quad V \mapsto \bigoplus_{\bfx \in \Ob(\C)} V(x)
\]
and another functor
\[
\pi: k\C \Mod \to \C \Mod, \quad M \mapsto (V: \bfx \mapsto 1_{\bfx} M)
\]
where $1_{\bfx}$ is the identity morphism on the object $\bfx$. Clearly, $\pi \circ \iota$ is isomorphic to the identity functor on $\C \Mod$, but $\iota \circ \pi$ is isomorphic to the identity functor on $k\C \Mod$ if and only if $\Ob(\C)$ is a finite set. In other words, a $\C$-module is a \textit{graded} $k\C$-module graded by $\Ob(\C)$.

The category algebra $k\C$ has a two-sided ideal $\mathfrak{I}$ which as a free $k$-module is spanned by all morphisms $f: \bfx \to \bfy$ such that $\bfx \neq \bfy$. Given a $\C$-module $V$, we define the zeroth homology group of $V$ by
\[
H_0(V) = k\C/\mathfrak{I} \otimes_{k\C} V
\]
which is a $\C$-module, and its value on $\bfx$ is isomorphic to
\[
V_{\bfx} / \sum_{\bfy < \bfx \atop f \in \C(\bfy, \bfx)} V_f(V_{\bfy}).
\]
Since $k\C/\mathfrak{I} \otimes_{k\C} -$ is right exact, for each $i \geqslant 0$, we can define the $i$-th homology group of $V$ by
\[
H_i(V) = \Tor_i^{k\C} (k\C/\mathfrak{I}, \, V).
\]
Correspondingly, the $i$-th homological degree with respect to a fixed norm $\| \bullet \|$ is defined to be
\[
\hd_i(V) = \sup \{ \| \bfx \| \mid \bfx \in \supp(H_i(V)) \}
\]
which might be infinity. We set $\hd_i(V)$ to be $-1$ when the above set is empty by convention. For brevity, we write $\gd(V)$ for $\hd_0(V)$ and $\pd(V)$ for $\max \{\hd_0(V), \, \hd_1(V) \}$, and call them the \textit{generation degree} and \textit{presentation degree}. The $\C$-module $V$ is said to be \textit{generated in finite degrees} (resp., \textit{presented in finite degrees}) if $\gd(V)$ is finite (resp., $\pd(V)$ is finite).

\begin{remark}
It is easy to see from the above definitions that if $V$ as a $\C$-module is generated by its values on objects in a subset $S \subseteq \Ob(\C)$, then $S$ contains $\supp(H_0(V))$ as a subset. It is also clear that every non-invertible morphism in $\C$ acts on $H_i(V)$ as the zero map for each $i \geqslant 0$.
\end{remark}

Now we turn to torsion theory of $\C$-modules. Given a $\C$-module $V$ and an object $\bfx \in \Ob(\C)$, an element $v \in V_{\bfx}$ is called \textit{torsion} if there is a morphism $f: \bfx \to \bfy$ such that $V_f(v) = 0 \in V_{\bfy}$. Since $\C$ satisfies the left Ore condition by Lemma \ref{facts}, it follows that all torsion elements in $V$ form a $\C$-submodule of $V$, called the \textit{torsion part} of $V$ and denoted by $V_T$. The quotient module $V/V_T$ is called the \textit{torsion free part} of $V$. We say that $V$ is \textit{torsion} if $V = V_T$, and \textit{torsion free} if $V_T = 0$. The \textit{torsion set} of $V$ along the $i$-th direction is
\[
\ts_i(V) = \{\bfx \in \Ob(\C) \mid \exists \, 0 \neq v \in V_{\bfx} \text{ such that there is a map } f \in \C(\bfx, \, \bfx + \bfo_i) \text{ satisfying } V_f(v) = 0 \},
\]
where $\bfo_i$ is the sequence of natural numbers such that the $i$-th entry is 1 and other entries are 0. Correspondingly, we define the \textit{torsion height} along the $i$-th direction of $V$ to be
\[
\tht_i(V) = \sup \{ x_i \mid \bfx \in \ts_i(V) \}.
\]
Again, if $\ts_i(V) = \emptyset$, we set $\tht_i(V) = -1$. The \textit{torsion height vector} is set to be
\[
\tht(V) = (\tht_1(V), \, \tht_2(V), \, \ldots).
\]

\begin{remark} \label{torsion height}
By definition, $\tht_i(V) = n$ if and only if for any morphism $f: \bfx \to \bfx + \bfo_i$ with $x_i > n$, the map $V_f: V_{\bfx} \to V_{\bfx + \bfo_i}$ is injective. We remind the reader that in general $\tht(V)$  might contain infinitely many nonzero entries.
\end{remark}

\section{Representations of $\mathbb{N}^{\infty}$}

In this section we let $k$ be a commutative ring and let $\mathscr{N}$ be the poset $\mathbb{N}^{\infty}$. As mentioned before, its representations are closely related to modules of the polynomial ring $A = k[X_1, \, X_2, \, \ldots]$ with countably many variables. Note that modules over $A$ coincide with representations of the quiver with one vertex and countably many loops such that every two loops commute. Since there is a natural quotient functor from $\mathscr{N}$ to this quiver (viewed as a category), every module over $A$ can be viewed as a representation of $\mathscr{N}$. Explicitly, given an $A$-module $M$, one can construct a representation $V$ of $\mathscr{N}$ by setting $V_{\bfx} = M$ for $\bfx \in \Ob(\mathscr{N})$, and the morphism $V_{\bfx} \to V_{\bfx+\bfo_i}$ for each $i \in \mathbb{N}_+$ can be defined by the action of $X_i$ on $M$. Conversely, given a $\mathscr{N}$-module $V$, one can define a multi-graded $A$-module
\[
M = \bigoplus_{\bfx \in \Ob(\N)} V_{\bfx},
\]
and the action of $X_i \in A$ on $M$ is given by maps $V_{\bfx} \to V_{\bfx + \bfo_i}$ componentwise. These two procedures are functorial, but are not inverse to each other even up to equivalence.

\subsection{Bounds for invariants with respect to the sup norm}

In this subsection we use the sup norm
\[
\| \bfx \| = \sup\{ x_i \mid i \in \mathbb{N} \}.
\]
The following simple observation plays a key role for us to establish upper bounds for homological invariants of $\N$-modules.

\begin{lemma} \label{key observation}
Let $V$ be an $\N$-module with $\gd(V) = n$, and let $\bfx$ be an object in $\N$. If $x_i \geqslant n$, then the map $V_h: V_{\bfx} \to V_{\bfx + \bfo_i}$ is surjective where $h$ is the unique morphism $\bfx \to \bfx + \bfo_i$.
\end{lemma}

\begin{proof}
The conclusion holds trivially if $\gd(V) = -1$ since in this case $V = 0$. Suppose that $n$ is nonnegative. Since $\| \bfx + \bfo_i \| \geqslant x_i + 1 > n$ and $\gd(V) = n$, it follows that $\bfx + \bfo_i$ is not contained in $\supp(H_0(V))$. Consequently, one has
\[
V_{\bfx + \bfo_i} = \sum_{\bfy < \bfx + \bfo_i, \, \|\bfy\| \leqslant n \atop f: \bfy \to \bfx + \bfo_i} V_f(V_{\bfy}).
\]
However, $\bfy < \bfx + \bfo_i$ and $\|\bfy\| \leqslant n$ imply that $\bfy \leqslant \bfx$. It follows that any such $f$ can be expressed as a composite $f = hg$ depicted as follows
\[
\xymatrix{
\bfy \ar[r]^-g & \bfx \ar[r]^-h & \bfx + \bfo_i.
}
\]
Thus $V_f(V_{\bfy}) = V_h (V_g (V_{\bfy})) \subseteq V_h(V_{\bfx})$, and hence $V_{\bfx + \bfo_i} = V_h(V_{\bfx})$ as claimed.
\end{proof}

As the first application of this lemma, we can use the first two homological degrees to bound torsion heights.

\begin{proposition} \label{bounding torsion height}
Let $V$ be an $\N$-module. Then for each $i \in \mathbb{N}_+$, one has
\[
\tht_i(V) \leqslant \pd(V) - 1.
\]
\end{proposition}

\begin{proof}
Without loss of generality we can assume that $V$ is nonzero and $\pd(V)$ is a finite integer $n$. If the conclusion is false, then by the definition of $\tht_i(V)$, we can find an object $\bfx$ such that $x_i \geqslant n$ and the map $V_h: V_{\bfx} \to V_{\bfx + \bfo_i}$ is not injective.

Consider a short exact sequence $0 \to W \to P \to V \to 0$ such that $P$ is a projective $\N$-module with $\gd(P) = \gd(V)$. It induces the following commutative diagram
\[
\xymatrix{
0 \ar[r] & W_{\bfx} \ar[r] \ar[d]^-{W_h} & P_{\bfx} \ar[r] \ar[d]^-{P_h} & V_{\bfx} \ar[r] \ar[d]^-{V_h} & 0\\
0 \ar[r] & W_{\bfx + \bfo_i} \ar[r] & P_{\bfx + \bfo_i} \ar[r] & V_{\bfx + \bfo_i} \ar[r] & 0.
}
\]
By Lemma \ref{facts}, every morphism in $\N$ is a monomorphism, so $P(\bfx) = k\N(\bfx, -)$ is torsion free for every $\bfx \in \Ob(\N)$. Thus we can assume that $P$ is torsion free, and hence $P_h$ is injective. But $V_h$ is not injective, so applying the snake lemma we conclude that $W_h$ is not surjective. By Lemma \ref{key observation}, we deduce that
\[
\gd(W) > x_i \geqslant n.
\]
On the other hand, a long exact sequence argument to the sequence $0 \to W \to P \to V \to 0$ tells us that
\[
\gd(W) \leqslant \max \{\gd(P), \, \hd_1(V) \} = \pd(V) = n.
\]
The conclusion follows by contradiction.
\end{proof}

The second application of Lemma \ref{key observation} gives an extremely simple bound of homological degrees for torsion free $\N$-modules.

\begin{lemma} \label{hd of torsion free}
Let $V$ be a torsion free $\N$-module. Then for every $s \geqslant 0$, one has $\hd_s(V) \leqslant \gd(V)$.
\end{lemma}

\begin{proof}
Without loss of generality we can assume that $\gd(V) = n$ is a natural number. It suffices to show the conclusion for $s = 1$. Indeed, if this is true, then take a short exact sequence $0 \to W \to P \to V \to 0$ such that $P$ is a torsion free projective $\N$-module with $\gd(P) = \gd(V)$. Applying the long exact sequence argument we deduce that
\[
\gd(W) \leqslant \max\{\gd(P), \, \hd_1(V) \} = \gd(V).
\]
Since $W$ is also torsion free, replacing $V$ by $W$ we conclude that
\[
\hd_2(V) = \hd_1(W) \leqslant \gd(W) \leqslant \gd(V).
\]
Then conclusion then follows by a recursive argument.

If the conclusion is false for $s = 1$, then by the previous argument, we have $\gd(W) \geqslant n+1$, so there exists an object $\bfx \in \supp(H_0(W))$ such that $x_i \geqslant n+1$ for a certain $i \in \mathbb{N}_+$. We consider the following commutative diagram
\[
\xymatrix{
0 \ar[r] & W_{\bfx - \bfo_i} \ar[r] \ar[d]^-{W_h} & P_{\bfx - \bfo_i} \ar[r] \ar[d]^-{P_h} & V_{\bfx - \bfo_i} \ar[r] \ar[d]^-{V_h} & 0\\
0 \ar[r] & W_{\bfx} \ar[r] & P_{\bfx} \ar[r] & V_{\bfx} \ar[r] & 0.
}
\]
Note that all vertical maps are injective since $V$ is torsion free. Furthermore, since $(\bfx - \bfo_i)_i = x_i - 1 \geqslant n$, Lemma \ref{key observation} tells us that $P_h$ and $V_h$ are surjective. Consequently, $W_h$ is surjective as well, so $(H_0(W))_{\bfx} = 0$ and hence $\bfx \notin \supp(H_0(W))$. The conclusion follows by contradiction.
\end{proof}

An application of the above result in commutative algebra is as follows.

\begin{corollary}
Let $A = k[X_1, \, X_2, \, \ldots]$ be the polynomial ring with countably many variables, and let $I \lhd A$ be an ideal generated by monomials such that the degree of each $X_i$ in these monomials is bounded by $n$. Then $I$ has a projective resolution $P^{\bullet} \to I \to 0$ such that each $P^j$ is a direct sum of principal ideals of $A$ generated by a monomial in which the degree of each $X_i$ is at most $n$.
\end{corollary}

\begin{proof}
We can identify $I$ as a submodule $V$ of the $\N$-module $k\N(\boldsymbol{0}, \, -)$ in an obviously way. By the above lemma, one can construct a projective resolution $Q^{\bullet} \to V \to 0$ of $\N$-modules such that $Q^i$ is isomorphic to a direct sum of $k\N(\bfx, \, -)$ such that $\| \bfx \| \leqslant n$ for each $i \in \mathbb{N}$. Note that each $k\N(\bfx, \, -)$ can be identified with the principal ideal of $A$ generated by $X_1^{x_1} X_2^{x_2} \ldots$ which is a monomial since $x_i = 0$ for $i \gg 0$. The conclusion follows from these observations.
\end{proof}

We generalize the conclusion of the above lemma to the general situation. Recall that an $\N$-module $V$ is presented in finite degrees if $\pd(V)$ is finite.

\begin{proposition} \label{bound hds}
Let $V$ be an $\N$-module. Then $\hd_s(V) \leqslant \pd(V)$ for $s \geqslant 0$. In particular, the category of $\N$-modules presented in finite degrees is abelian.
\end{proposition}

\begin{proof}
The inequality holds trivially for $s \leqslant 1$. Take a short exact sequence $0 \to W \to P \to V \to 0$ such that $P$ is a torsion free projective $\N$-module and $\gd(P) = \gd(V)$. For $s \geqslant 2$, by Lemma \ref{hd of torsion free} one has
\[
\hd_s(V) = \hd_{s-1}(W) \leqslant \gd(W) \leqslant \max\{\gd(V), \, \hd_1(V) \} = \pd(V),
\]
so the first statement is valid.

Let $0 \to U \to V \to W \to 0$ be a short exact sequence of $\N$-modules. By the induced long exact sequence of homology groups, one has
\[
\hd_s(V) \leqslant \max \{\hd_s(U), \, \hd_s(W) \}
\]
for $s \geqslant 0$. Consequently, if both $U$ and $W$ are presented in finite degrees, then their homological degrees are all finite. By the above inequality, homological degrees of $V$ are also finite, and in particular $V$ is presented in finite degrees. Using the same argument, one can show that if any two of the three terms in this sequence are presented in finite degrees, so is the third one. The second statement follows from this observation.
\end{proof}

\subsection{Bounds for invariants with respect to the sum norm}

It might seem to the reader that some results established in previous subsection with respect to the sup norm are too strong to be interesting, for instance Proposition \ref{bound hds}. The reason is that the sup norm is too weak, and hence too many distinct objects share the same norm. For example, objects in the following set
\[
\{ \bfo_1 + \bfo_2 + \ldots + \bfo_s \mid s \in \mathbb{N}_+ \}
\]
all have norm 1, which seems quite unnatural. Indeed, this norm does not respect the poset structure of $\N$; that is, one does not have $\| \bfx \| < \| \bfy \|$ when $\bfx < \bfy$.

In this subsection we consider the more reasonable sum norm. Throughout this subsection all homological degrees are defined with respect to the sum norm, but we use the same notation. Hopefully this would not cause too much confusion to the reader. Unfortunately, the situation becomes much more complicated. For instance, in contrast to Proposition \ref{bound hds}, the category of $\N$-modules presented in finite degrees with respect to the sum norm is no longer abelian, as illustrated by the following example.

\begin{example} \label{counterexample}
Let $k$ be a field and $P$ be the direct sum of countably many copies of $P(\boldsymbol{0})$. Choose a basis $\{b_1, \, b_2, \, \ldots \}$ for $P_{\boldsymbol{0}}$. For each $\bfx \in \Ob(\N)$, $P_{\bfx}$ has a basis consisting of elements $P_f(b_i)$, $i \geqslant 0$, where $f$ is the unique morphism from $\boldsymbol{0}$ to $\bfx$. By abuse of notation we denote $V_f(b_i)$ by $b_i$ again. Define a submodule $V$ of $P$ as follows:
\begin{itemize}
\item $V_{\boldsymbol{0}} = 0$ and
\[
V_{\bfo_s} =
\begin{cases}
\langle b_{(i+1)/2} \rangle, & \text{if $s = 2i - 1$ is odd};\\
\langle b_{p_i} + b_{{p_i}^2}  + \ldots + b_{p_i^{i+1}} \rangle; & \text{if $s = 2i$ is even}
\end{cases}
\]
where $p_i$ is the $i$-th prime number;

\item $V$ is generated in degree 1.
\end{itemize}
For instance, we have $V_{\bfo_1} = \langle b_1 \rangle$, $V_{\bfo_2} = \langle b_2 + b_4 \rangle$, $V_{\bfo_3} = \langle b_2 \rangle$, and $V_{\bfo_4} = \langle b_3 + b_9 + b_{27} \rangle$.

Clearly, both $P$ and $P/V$ are presented in finite degrees. However, this is not the case for the submodule $V$ of $P$. Indeed, we can construct a short exact sequence
\[
0 \to K \to \bigoplus_{s \in \mathbb{N}_+} P(\bfo_s) \to V \to 0.
\]
A direct computation shows that for any $s \in \mathbb{N}_+$, we can find an object $\bfx$ with sum norm $p_s + 1$ such that $\bfx \in \supp(H_0(K))$. For examples:
\begin{itemize}
\item $s = 1$, $p_1 = 2$, $\bfx = (0, \, 1, \, 1, \, 0, \, 0, \, 0, \, 1, \, 0, \, \ldots)$, $\dim_k (H_0(K))_{\bfx} = 1$;
\item $s = 2$, $p_2 = 3$, $\bfx = \bfo_4 + \bfo_5 + \bfo_{17} + \bfo_{53}$, $\dim_k (H_0(K))_{\bfx} = 1$.
\end{itemize}
Thus $\gd(K) = \infty$, and hence $V$ is not presented in finite degrees.
\end{example}

This example tells us that it is impossible to bound higher homological degrees (even the second homological degree) of $\N$-modules $V$ presented in finite degrees in terms of its presentation degree. However, if $k$ is a field and $V$ is moreover finitely generated\footnote{The condition that $V$ is finitely generated and presented in finite degrees does not imply that $V$ is finitely presented.}, then its second homological degree is finite. In the rest of this section we give a proof for this fact.

\begin{lemma}  \label{submodules of free module}
Let $k$ be a field and let $V$ be a submodule of $P(\bfx)$ for a certain $\bfx \in \Ob(\N)$. Then
\[
\hd_1(V) \leqslant 2 \gd(V).
\]
In particular, if $V$ is finitely generated, then it is finitely presented.
\end{lemma}

\begin{proof}
Without loss of generality we assume that $V$ is nonzero, $\gd(V) = n$ is finite, and $\bfx = \boldsymbol{0}$ since $P(\bfx)$ can be identified with a submodule of $P(\boldsymbol{0})$. Let $S = \supp (H_0(V))$, which is nonempty, and let $V(\bfx)$ be the submodule of $V$ generated by $V_{\bfx}$. Then
\[
V = \sum_{\bfx \in S} V(\bfx).
\]
Since $V(\bfx)$ is torsion free, it is actually isomorphic to $P(\bfx)$. Thus we obtain a surjection
\[
Q = \bigoplus_{\bfx \in S} V(\bfx) \longrightarrow \sum_{\bfx \in S} V(\bfx) = V,
\]
and hence a short exact sequence $0 \to K \to Q \to V \to 0$, from which we deduce that $\hd_1(V) \leqslant \gd(K)$.

Now we investigate $K$. Fix $\bfz \in \Ob(\N)$ and let
\[
T = S \cap \{ \bfx \in \Ob(\N) \mid \bfx \leqslant \bfz \}
\]
which is a finite set. We have
\[
K_{\bfz} = \{(V_{\bfx \to \bfz} (v_{\bfx}))_{\bfx \in T} \mid v_{\bfx} \in V_{\bfx}, \, \sum_{\bfx \in T} V_{\bfx \to \bfz} (v_{\bfx}) = 0 \},
\]
where $\bfx \to \bfz$ is the morphism from $\bfx$ to $\bfz$. Since $V$ is a submodule of $P = P(\boldsymbol{0})$, after choosing a nonzero element $v \in P_{\boldsymbol{0}}$, we have $v_{\bfx} = \lambda_{\bfx} P_{\boldsymbol{0} \to \bfx}(v)$ for a certain $\lambda_{\bfx} \in k$ and hence $V_{\bfx \to \bfz}(v_{\bfx}) = \lambda_{\bfx} P_{\boldsymbol{0} \to \bfz} (v)$. Consequently,
\[
K_{\bfz} = \{(\lambda_{\bfx} P_{\boldsymbol{0} \to \bfz} (v))_{\bfx \in T} \mid \sum_{\bfx \in T} \lambda_{\bfx} = 0 \}.
\]
Clearly, this space is spanned by elements of the following form:
\[
u = (0, \, \ldots, \, 0, \, P_{\boldsymbol{0} \to \bfz} (v) = V_{\bfx \to \bfz} (\frac{v_{\bfx}}{\lambda_{\bfx}}), \, 0, \, \ldots, \, 0, \, -P_{\boldsymbol{0} \to \bfz} (v) = - V_{\bfy \to \bfz} (\frac{v_{\bfy}}{\lambda_{\bfy}}), \, 0, \ldots, \, 0)
\]
where $\bfx$ and $\bfy$ are two distinct elements in $T$ and $\lambda_{\bfx} \neq 0 \neq \lambda_{\bfy}$. However, by the commutative diagram
\[
\xymatrix{
\boldsymbol{0} \ar[r] \ar[d] & \bfy \ar[ddr] \ar[d]\\
\bfx \ar[r] \ar[drr] & \bfx \vee \bfy \ar[dr]\\
 & & \bfz
}
\]
We know that
\[
w = (0, \, \ldots, \, 0, \, V_{\bfx \to \bfx \vee \bfy} (\frac{v_{\bfx}}{\lambda_{\bfx}}), \, 0, \, \ldots, \, 0, \, -V_{\bfy \to \bfx \vee \bfy} (\frac{v_{\bfy}}{\lambda_{\bfy}}), \, 0, \ldots, \, 0) \in K_{\bfx \vee \bfy},
\]
where $\bfx \vee \bfy$ is the least upper bound of $\bfx$ and $\bfy$. Furthermore, $K_{\bfx \vee \bfy \to \bfz}$ sends $w$ to $u$. Consequently, $K_{\bfz}$ can be generated by the values of $K$ on objects $\bfx \vee \bfy$ with $\bfx, \, \bfy \in T$. Since $\| \bfx \vee \bfy \| \leqslant \|\bfx \| + \|\bfy\| \leqslant 2n$, the desired inequality follows.

If $V$ is finitely generated (or equivalently, $V$ is generated by its values on finitely many objects and the values on these objects are finite dimensional), then $S$ is a finite set, and for each $\bfx \in S$, the dimension of $V_{\bfx}$ is finite. Consequently, the dimensions of $Q_{\bfx}$ and $K_{\bfx}$ for every $\bfx \in \Ob(\N)$ are finite as well. Furthermore, there are only finitely many $\bfx \vee \bfy$ with $\bfx, \bfy \in S$. Since $K$ can be generated by its values on these objects, we deduce that $K$ is finitely generated, and hence $V$ is finitely presented.
\end{proof}

\begin{remark}
Many key ingredients in above proof heavily rely on the extra condition that $k$ is a field. For instances, $V(\bfx) \cong P(\bfx)$ in the first paragraph of the proof does not hold for an arbitrary commutative ring. Similarly, in general we cannot find a basis for $K_{\bfz}$ in terms of $u$ in the second paragraph since $\lambda_{\bfx}$ might not be invertible in $k$, and since $K_{\bfz}$ might not be a free $k$-module.
\end{remark}

\begin{remark} \label{left coherence of kN}
The second statement of this lemma means that the category algebra $k\N$ is \textit{locally left coherent}. Consequently, finitely presented $\N$-modules coincide with coherent modules, and these modules form an abelian subcategory of $\N \Mod$. This result can be deduced from the well known fact that the polynomial ring $A = k[X_1, \, X_2, \, \ldots]$ is a coherent ring and $\N \Mod$ can be viewed as a full subcategory of $A \Mod$. However, the inequality in this lemma holds in a more general setup; that is, we do not need to impose the condition that $V$ is generated by its values on finitely many objects, or the even stronger condition that $V$ is finitely generated.
\end{remark}

\begin{proposition} \label{coherence}
Let $k$ be a field, $V$ an $\N$-module and $\omega = \dim_k(H_0(V))$. Then
\[
\hd_2(V) \leqslant 2^{\omega} \hd_1(V).
\]
\end{proposition}

\begin{proof}
Without loss of generality one can assume that both $\omega$ and $\hd_1(V)$ are finite. One can find a short exact sequence $0 \to W \to P \to V \to 0$ such that $P$ is a coproduct of $\omega$ copies of $\N$-modules of the form $P(\bfx)$. Furthermore, one can assume that $H_0(P) \cong H_0(V)$, and hence $\gd(W) = \hd_1(V)$.

When $\omega = 1$, by the previous lemma, one has
\[
\hd_2(V) = \hd_1(W) \leqslant 2 \gd(W) = 2\hd_1(V),
\]
so the conclusion holds. Now assume that $\omega \geqslant 2$.

By taking a copy $P' = P(\bfx)$ inside $P$ and letting $W' = P' \cap W$, we get a commutative diagram of short exact sequences
\[
\xymatrix{
0 \ar[r] & W' \ar[r] \ar[d] & P' \ar[r] \ar[d] & V' \ar[r] \ar[d] & 0\\
0 \ar[r] & W \ar[r] \ar[d] & P \ar[r] \ar[d] & V \ar[r] \ar[d] & 0\\
0 \ar[r] & W'' \ar[r] & P'' \ar[r] & V'' \ar[r] & 0.
}
\]
Since $\dim_k(H_0(V'')) \leqslant \dim_k(H_0(P'')) = \omega -1$ and $\hd_1(V'') \leqslant \gd(W'') \leqslant \gd(W) = \hd_1(V)$, the induction hypothesis tells us that
\[
\hd_2(V'') \leqslant 2^{\omega - 1} \hd_1(V'') \leqslant 2^{\omega - 1} \gd(W'') \leqslant 2^{\omega - 1} \hd_1(V).
\]
By looking at the first column, we deduce that
\[
\hd_1(V') \leqslant \gd(W') \leqslant \max\{\gd(W), \, \hd_1(W'') \} \leqslant \max\{\gd(W), \, 2^{\omega - 1} \hd_1(V) \} = 2^{\omega - 1} \hd_1(V).
\]
Again, by the established conclusion for $\omega = 1$, we have $\hd_2(V') \leqslant 2 \hd_1(V') \leqslant 2^{\omega} \hd_1(V)$. By looking at the last column, we get the desired upper bound.
\end{proof}

\subsection{Shift functors}

In this subsection we introduce shift functors and describe some elementary properties. Throughout this subsection let $k$ be a commutative ring and all numerical invariants are defined with respect to the sup norm.

Note that for each positive integer $i$ there is a functor $\iota_i: \N \to \N$ sending $\bfx \in \Ob(\N)$ to $\bfx + \bfo_i$, and the unique morphism $\bfx \to \bfy$ to the unique morphism $\bfx+\bfo_i \to \bfy+\bfo_i$. The shift functor $\Sigma_i: \N \Mod \to \N \Mod$ along the $i$-th direction is the pullback of $\iota_i$, mapping an $\N$-module $V$ to $V \circ \iota_i$. The following lemma collects some elementary properties of $\Sigma_i$.

\begin{lemma} \label{properties of shift functor}
Let $\Sigma_i$ be the $i$-th shift functor. Then:
\begin{enumerate}
\item $\Sigma_i$ is exact;

\item $\Sigma_i \circ \Sigma_j = \Sigma_j \circ \Sigma_i$;

\item there is a natural transformation $\phi^i$ between the identity functor on $\N \Mod$ and $\Sigma_i$;

\item the kernel $K_iV$ of the map $V \to \Sigma_iV$ has the following description
\[
K_iV = \bigoplus_{\bfx \in \Ob(\N)} \{v \in V_{\bfx} \mid v \text{ is sent to 0 by the map } \bfx \to \bfx + \bfo_i \}.
\]

\item the natural maps $V \to \Sigma_iV$ is injective for each $i \in \mathbb{N}$ if and only if $V$ is torsion free;

\item for $P(\bfx) = k\N(\bfx, -)$, one has
\[
\Sigma_i P(\bfx) \cong
\begin{cases}
P(\bfx - \bfo_i), & \text{if } x_i \geqslant 1;\\
P(\bfx), & \text{else.}
\end{cases}
\]
\end{enumerate}
\end{lemma}

\begin{proof}
The first two statements follow readily from the definition. The third statement follows from the existence of a natural transformation between the identity functor on $\N$ and $\iota_i$, illustrated by the diagram
\[
\xymatrix{
\bfx \ar[r]^f \ar[d] & \bfy \ar[d]\\
\iota_i \bfx = \bfx+\bfo_i \ar[r]^-{\iota_i f} & \iota_i \bfy = \bfy+\bfo_i.
}
\]
This diagram also implies the forth statement, which in turn applies the fifth statement. The last statement holds by a very simple observation.
\end{proof}

Now we construct the shift functor along all directions, which intuitively can be viewed as an infinite composite of $\Sigma_i$ for all $i \in \mathbb{N}$. By statement (3) of Lemma \ref{properties of shift functor}, one can construct a sequence
\[
\xymatrix{
V \ar[r]^-{\phi_V^1} & \Sigma_1 V \ar[r]^-{\phi^2_{\Sigma_1 V}} & \Sigma_{[2]} V \ar[r]^-{\phi^3_{\Sigma_{[2]}V}} & \Sigma_{[3]}V \ar[r] & \ldots,
}
\]
where $\Sigma_{[n]} = \Sigma_n \circ \Sigma_{n-1} \circ \ldots \circ \Sigma_1$ for $n \geqslant 1$. Taking the colimit of this diagram we obtain an $\N$-module. Clearly, the rule assigning $V$ to this colimit is functorial, so induces an endofunctor in $\N \Mod$ denoted by $\BS$. Furthermore, there is a natural map from $V$ to $\BS V$ such that the following diagram commutes
\begin{equation} \label{diagram}
\xymatrix{
V \ar[r]^-{\phi^1_V} \ar[dr]_-{\psi} & \Sigma_1 V \ar[r]^-{\phi^2_{\Sigma_1V}} \ar[d]^-{\psi^1} & \Sigma_{[2]} V \ar[r] \ar[dl]^-{\psi^2} & \ldots\\
 & \BS V.
}
\end{equation}

We describe some useful facts about $\BS$ in the following lemma.

\begin{lemma} \label{sigma infinity}
Suppose that $V$ is a nonzero $\N$-module. One has:
\begin{enumerate}
\item $\BS$ is an exact functor;

\item $\gd(\BS V) \leqslant \gd(V)$, and $\gd(\BS V) \leqslant \gd(V) - 1$ when $\gd(V) \geqslant 1$;

\item for $i \in \mathbb{N}_+$, one has $\tht_i(\BS V) \leqslant \tht_i(V) - 1$ whenever $\tht_i(V) \geqslant 0$;

\item the map $\psi$ is injective if and only if $V$ is torsion free;

\item if $V$ is torsion free, so is $\BS V$.

\item if $V$ is torsion, so is $\BS V$.
\end{enumerate}
\end{lemma}

\begin{proof}
All facts follow from standard arguments about direct limits.

(1) This is clear since $\BS$ is the direct limit of a sequence of exact functors, and hence is exact.

(2) First we prove the conclusion for the special case that $V = P(\bfx)$. Let $N$ be the minimal integer such that $x_i = 0$ for $i > N$. By (6) of Lemma \ref{properties of shift functor}, $\Sigma_{[i]} V \cong P(\bfy)$ for $i \geqslant N$ where
\[
y_j =
\begin{cases}
x_j, & \text{if } x_j = 0;\\
x_j - 1, & \text{if } x_j \geqslant 1.
\end{cases}
\]
Furthermore, all but finitely many maps appearing in diagram (3.1) are isomorphisms. Consequently, $\BS P(\bfx) \cong P(\bfy)$, and $\|\bfy\| = \|\bfx\| -1$ whenever $\|\bfx\| \geqslant 1$, so the conclusion holds for this special case.

Now consider the general case. Take a surjection
\[
P = \bigoplus_{\bfx \in \supp(H_0(V))} P(\bfx)^{c_{\bfx}} \to V
\]
where the multiplicity $c_{\bfx}$ might be infinity. It is clear that $\gd(P) = \gd(V)$. Applying $\BS$ to it we obtain a surjection $\BS P \to \BS V$. If $\gd(V) \geqslant 1$, the conclusion for the special case asserts that $\gd(\BS P) = \gd(P) - 1$, so
\[
\gd(\BS V) \leqslant \gd(\BS P) = \gd(P) - 1 = \gd(V) - 1
\]
as desired.

(3) Without loss of generality we can assume that $\tht_i(V) = n \geqslant 0$ is finite. By Remark \ref{torsion height}, for each object $\bfx$ with $x_i \geqslant n+1$, the map $V_f: V_{\bfx} \to V_{\bfx + \bfo_i}$ is injective. Then when $s \geqslant i$, for any object $\bfy$ with $y_i \geqslant n$, the map $(\Sigma_{[s]} V)_{\bfy} \to (\Sigma_{[s]} V)_{\bfy + \bfo_i}$ is injective by Remark \ref{torsion height} since $(\Sigma_{[s]} V)_{\bfy} = V_{\bfz}$ where
\[
z_j = \begin{cases}
y_j + 1, & j \leqslant s;\\
y_j, & j \geqslant s+1,
\end{cases}
\]
and in particular $z_i = y_i + 1 \geqslant n+1$. Taking direct limit for the commutative diagram
\[
\xymatrix{
V_{\bfy} \ar[r] \ar[d] & (\Sigma_1 V)_{\bfy} \ar[r] \ar[d] & (\Sigma_{[2]} V)_{\bfy} \ar[r] \ar[d] & \ldots\\
V_{\bfy + \bfo_i} \ar[r] & (\Sigma_1 V)_{\bfy + \bfo_i} \ar[r] & (\Sigma_{[2]} V)_{\bfy + \bfo_i} \ar[r] & \ldots\\
}
\]
and noting that all but finitely many vertical maps are injective, we conclude that the map $(\BS V)_{\bfy} \to (\BS V)_{\bfy + \bfo_i}$ is also injective. Again, by Remark \ref{torsion height}, it follows that $\tht_i(\BS V) \leqslant \tht_i(V) - 1$.

(4) If $V$ is not torsion free, then by statement (5) of Lemma \ref{properties of shift functor}, there is a certain $i \in \mathbb{N}_+$ such that the natural map $V \to \Sigma_i V$ is not injective; that is, $\tht_i(V) \geqslant 0$. It is easy to check that $\tht_i(\Sigma_{[i]} V) < \tht_i(V)$. Consequently, in the top row of diagram (\ref{diagram}), the composite map from $V$ to $\Sigma_{[i]} V$ is not injective, and hence the map $\psi$ cannot be injective.

Conversely, if $V$ is torsion free, then all maps in the top row of diagram (\ref{diagram}) are injective. If $\psi$ is not injective, then one may find an object $\bfx$ and a certain $0 \neq v \in V_{\bfx}$ such that $v$ eventually becomes 0 by maps in the top row, which is impossible.

(5) This is clear. Indeed, if $V$ is torsion free, then by the previous statement, the natural map $V \to \BS V$ is injective. Since $\BS$ is exact, the natural map $\BS V \to \BS (\BS V)$ is also injective. Again by the previous statement, $\BS V$ is torsion free.

(6) For each object $\bfx \in \Ob(\N)$ as well as an arbitrary element $v \in (\BS V)_{\bfx}$, we want to show that $v$ is torsion. By definition, $(\BS V)_{\bfx}$ is the direct limit of the following sequence
\[
V_{\bfx} \to (\Sigma_1 V)_{\bfx} \to (\Sigma_{[2]} V)_{\bfx} \to \ldots.
\]
Then for a sufficiently large $n$, we can find a representative $\tilde{v} \in (\Sigma_{[n]} V)_{\bfx}$ of $v$. It is easy to see that $\Sigma_{[n]} V$ is torsion, so we can find another object $\bfy$ such that the map $f: \bfx \to \bfy$ sends $\tilde{v}$ to $0 \in (\Sigma_{[n]} V)_{\bfy}$. By looking at the commutative diagram
\[
\xymatrix{
(\Sigma_{[n]} V)_{\bfx} \ar[r] \ar[d] & (\Sigma_{[n]} V)_{\bfy} \ar[d]\\
(\BS V)_{\bfx} \ar[r] & (\BS V)_{\bfy},
}
\]
we conclude that $v$ is sent to 0 by $f$ as well.
\end{proof}

\begin{remark}
In general one can not expect that $\gd(\BS V) = \gd(V) - 1$ if $\gd(V) \geqslant 1$. For example, let $V$ be a torsion free module such that $V_{\bfx} = 0$ if $\| \bfx \| \leqslant 1$, and $V_{\bfx} = k$ otherwise. Then $\gd(V) \geqslant 2$ since clearly the object $(2, \, 0, \, 0, \ldots)$ is contained in $\supp(H_0(V))$. However, a direct computation shows that $\BS V \cong k\N(\boldsymbol{0}, -)$, so $\gd(\BS V) = 0$.
\end{remark}

\begin{remark} \label{kernel}
Let $\boldsymbol{K}$ be the kernel functor induced by the natural transformation $\psi$ in diagram (\ref{diagram}). It is easy to see that an element $v \in V_{\bfx}$ lies in $(\boldsymbol{K} V)_{\bfx}$ if there is a non-invertible morphism $f: \bfx \to \bfy$ such that $x_i \leqslant y_i \leqslant x_i+1$ for $i \in \mathbb{N}_+$ and $V_f(v) = 0$.
\end{remark}

Now we introduce a type of special $\N$-modules, which is a generalization of projective $\N$-modules.

\begin{definition}
Let $\N_{\bfx}$ be the subcategory of $\N$ consisting of the single object $\bfx$. A $\N$-module $V$ is said to be induced at $\bfx$ if $V$ is isomorphic to the left Kan extension of $V_{\bfx}$ with respect to the inclusion functor $\N_{\bfx} \to \N$; that is, $V \cong k\N \otimes_{k\N_{\bfx}} V_{\bfx}$.
\end{definition}

If $V$ is induced at $\bfx$, then for $\bfy \in \Ob(\N)$, $V_{\bfy} \cong V_{\bfx}$ if $\bfy \geqslant \bfx$, and $V_{\bfy} = 0$ otherwise. Correspondingly, each map $V_f$ is either an isomorphism or the zero map.

\begin{proposition}
Let $V$ be an $\N$-module presented in finite degrees. Then for $s \geqslant \pd(V)$, $\BS^s V$ is either the zero module or induced at $\boldsymbol{0}$, where $\BS^s$ is the composite of $s$ copies of $\BS$.
\end{proposition}

\begin{proof}
When $V$ is torsion, $\BS^s V$ is also torsion by statement (6) of Lemma \ref{sigma infinity}. However, by statement (3) of this lemma as well as Proposition \ref{bounding torsion height}, we know that $\tht_i(\BS^s V) = -1$ for all $i \geqslant 1$; that is, $\BS^s V$ is torsion free. This happens if and only if $\BS^s V$ is the zero module.

If $V$ is not torsion free, we consider the short exact sequence $0 \to V_T \to V \to V_F \to 0$ where $V_T$ is the torsion part and $V_F$ is the torsion free part. Applying $\BS^s$ to it and using the same argument as in the previous paragraph we deduce that $\BS^s V_T = 0$, so $\BS^s V \cong \BS^s V_F$, which is torsion free by statement (5) of Lemma \ref{sigma infinity}, and whose generation degree is at most 0 by statement (2) of Lemma \ref{sigma infinity}. Since $\BS^s V_F$ cannot be the zero module, it follows that it is a torsion free $\N$-module with generation degree 0, which is clearly induced at $\boldsymbol{0}$.
\end{proof}

\begin{corollary}
If $V$ is a torsion free $\N$-module presented in finite degrees, then it can be embedded into an $\N$-module induced at $\boldsymbol{0}$.
\end{corollary}

\begin{proof}
Let $s = \pd(V)$. The composite of the following sequence is an injective map
\[
V \to \BS V \to \BS^2 V \to \ldots \to \BS^s V
\]
with $\BS^s V$ the desired $\N$-module induced at $\boldsymbol{0}$.
\end{proof}

In the rest of this subsection we want to generalize the above results to arbitrary $\N$-modules. For this purpose, we need to introduce the ``infinite product" of $\BS$. Given an $\N$-module $V$, we have a sequence of natural maps
\[
\xymatrix{
V \ar[r]^-{\psi_V} & \BS V \ar[r]^-{\psi_{\BS V}} & \BS^2 V \ar[r] & \ldots.
}
\]
We denote the colimit of this sequence in $\N \Mod$ by $\BS^{\infty} V$. Clearly, this is a functorial construction, and there is a natural map $\boldsymbol{\psi}_V: V \to \BS^{\infty} V$. Moreover, $\BS^{\infty}$ is an exact functor.

\begin{remark}
We construct $\BS^{\infty}$ as the direct limit of $\BS$, which itself is a direct limit. Actually, one can define $\BS^{\infty}$ in one step. Given an $\N$-module $V$, we define a functor $\boldsymbol{V}: \N \to \N \Mod$ as follows: for $\bfx = (x_1, \, \ldots \, x_r, \, 0, \, 0, \, \ldots) \in \Ob(\N)$, define $\boldsymbol{V}_{\bfx} = \Sigma_1^{x_1} \ldots \Sigma_r^{x_r} V$. When $\bfx \leqslant \bfy$, the morphism from $\boldsymbol{V}_{\bfx}$ to $\boldsymbol{V}_{\bfy}$ is given by natural transformations between the functor $\Sigma_1^{x_1} \ldots \Sigma_r^{x_r}$ and the functor $\Sigma_1^{y_1} \ldots \Sigma_r^{y_s}$, which is well defined by statement (2) of Lemma \ref{properties of shift functor}. One can check that the colimit of $\boldsymbol{V}$ in $\N \Mod$ is precisely $\BS^{\infty} V$. We will give one more description for $\BS^{\infty}$ in Proposition \ref{sheafification}.
\end{remark}

\begin{proposition} \label{Sigma infinity}
Let $V$ be an $\N$-module. One has:
\begin{enumerate}
\item If $V$ is torsion, then $\BS^{\infty} V = 0$; otherwise, $\BS^{\infty} V$ is induced at $\boldsymbol{0}$.

\item If $V$ is torsion free, then it can be embedded into $\BS^{\infty} V$.
\end{enumerate}
\end{proposition}

\begin{proof}
(1) Suppose that $V$ is torsion. Then for any $\bfx \in \Ob(\N)$ and any $v \in (\BS^{\infty} V)_{\bfx}$, one can find a representative $\tilde{v} \in (\BS^r V)_{\bfx}$ for a large enough $r$. By statement (6) of Lemma \ref{sigma infinity}, $\BS^r V$ is also torsion. In particular, we can find a morphism $f: \bfx \to \bfy$ such that $f$ sends $\tilde{v}$ to 0. Let
\[
s = \max\{ y_i - x_i \mid i \in \mathbb{N}_+ \}.
\]
By Remark \ref{kernel}, the kernel on $\bfx$ of the composite of the following maps
\[
\BS^r V \to \BS^{r+1} V \to \ldots \to \BS^{r+s} V
\]
consists of $u \in (\BS^r V)_x$ satisfying the following condition: there is an object $\bfz \in \Ob(\N)$ such that $z_i \leqslant x_i + s$ for $i \geqslant 1$ and the morphism $g: \bfx \to \bfz$ sends $u$ to 0. Therefore, $\tilde{v}$ is contained in this kernel, so $\tilde{v}$ becomes 0 in $(\BS^{r+s}V)_{\bfx}$. Thus $v = 0$ and hence $\BS^{\infty} V$ is the zero module.

In the case that $V$ is not torsion, it is easy to see that $\BS^{\infty} V \cong \BS^{\infty} V_F$, where $V_F$ is the torsion free part of $V$. Thus without loss of generality we can assume that $V$ is a nonzero torsion free $\N$-module. In this case, natural maps in the following sequence
\[
V \to \BS V \to \BS^2 V \to \ldots
\]
are all injective and all terms in this sequence are nonzero torsion free $\N$-modules by statements (4) and (5) of Lemma \ref{sigma infinity}. Consequently, $\BS^{\infty} V$ is a nonzero torsion free $\N$-module.

It remains to show that $\gd(\BS^{\infty} V) = 0$. But this is clear. Indeed, for the special case that $V = P(\bfx)$, we have $\BS^{s} V \cong P(\boldsymbol{0})$ for $s \geqslant \| \bfx \|$, so $\BS^{\infty} V \cong P(\boldsymbol{0})$. Taking a surjective homomorphism
\[
\bigoplus_{\bfx \in \Ob(\N)} P(\bfx)^{c_{\bfx}} \to V
\]
and applying $\BS^{\infty}$ we have a surjective homomorphism
\[
\BS^{\infty} (\bigoplus_{\bfx \in \Ob(\N)} P(\bfx)^{c_{\bfx}}) \cong \bigoplus_{\bfx \in \Ob(\N)}(\BS^{\infty} P(\bfx))^{c_{\bfx}} \to \BS^{\infty} V.
\]
Consequently, $\gd(\BS^{\infty} V) \leqslant 0$. Since $\BS^{\infty} V \neq 0$, it follows that $\gd(\BS^{\infty} V) = 0$.

(2) Since all maps in the sequence $V \to \BS V \to \BS^2 V \to \ldots$ are injective, the induced map $V \to \BS^{\infty} V$ is injective as well and gives the desired embedding.
\end{proof}

\begin{remark} \label{cokernel is torsion}
It is easy to see that $\BS^{\infty} V = 0$ if and only if $V$ is torsion. Moreover, the kernel of the map $\boldsymbol{\psi}: V \to \BS^{\infty} V$ coincides with the torsion part $V_T$ of $V$.
\end{remark}

\section{Homological degrees of representations of $\mathbb{N}^{\infty}$-type combinatorial categories}

Let $\C$ be an arbitrary combinatorial category of type $\mathbb{N}^{\infty}$, and $k$ a commutative ring. In this section we give an upper bound for homological degrees of $\C$-modules in terms of their first two homological degrees, generalizing Proposition \ref{bound hds}.

Denote by $\N$ the poset $\mathbb{N}^{\infty}$, and choose a functor $\iota: \N \to \C$ such that $\iota$ is a bijection restricted to object sets. The functor $\iota: \N \to \C$ induces a restriction functor $\downarrow_{\N}^{\C}: \C \Mod \to \N \Mod$ via pullback. The following observation plays a key role for us to prove the main result in this section.

\begin{lemma} \label{res preserves projectives}
Given $\bfx \in \Ob(\C)$, let $P(\bfx) = k\C(\bfx, -)$, $Q(\bfx) = k\N(\bfx, -)$, and $G_{\bfx} = \C(\bfx, \, \bfx)$. Then
\[
P(\bfx) \downarrow_{\N}^{\C} \cong Q(\bfx)^{|G_{\bfx}|},
\]
where $|G_{\bfx}|$ is the order of $G_{\bfx}$.
\end{lemma}

\begin{proof}
Note that $P(\bfx)$ as a free $k$-module has a basis consisting of morphisms in $\C(\bfx, -)$. Without loss of generality we can view $\N$ as a subcategory of $\C$, so $\N(\bfx, -)$ is a subset of $\C(\bfx, -)$. For $g \in \C(\bfx, \bfx)$, define
\[
\N(\bfx, -) g = \{fg \mid f \in \N(\bfx, -) \}.
\]
Since $G_{\bfx}$ acts regularly on $\C(\bfx, \bfy)$ for any $\bfy \in \Ob(\C)$, it follows that $\C(\bfx, -)$ is a disjoint union of $\N(\bfx, -) g$ with $g$ ranging over all elements in $G_{\bfx}$. The conclusion then follows from this observation.
\end{proof}

\begin{lemma} \label{compare hds}
Let $V$ be a $\C$-module, and $\| \bullet \|$ be either the sup norm or the sum norm on $\Ob(\C)$. Then $\gd(V) = \gd(V \downarrow_{\N}^{\C})$ and $\pd(V) = \pd(V \downarrow_{\N}^{\C})$.
\end{lemma}

\begin{proof}
Take a short exact sequence $0 \to W \to P \to V \to 0$ such that $P$ is a projective $\C$-module and $\gd(P) = \gd(V)$. Applying the restriction functor we get the following short exact sequence of $\N$-modules
\[
0 \to W \downarrow_{\N}^{\C} \to P \downarrow_{\N}^{\C} \to V \downarrow_{\N}^{\C} \to 0,
\]
which implies that
\[
\gd(V \downarrow_{\N}^{\C}) \leqslant \gd(P \downarrow_{\N}^{\C}) = \gd(P) = \gd(V)
\]
where the first equality comes from Lemma \ref{res preserves projectives}. On the other hand, for any object $\bfx \in \Ob(\C)$ such that $\bfx \in \supp(H_0(V))$, we know that
\[
\sum_{\bfy < \bfx \atop f \in \C(\bfy, \bfx)} V_f(V_{\bfy}) \subsetneqq V_{\bfx}.
\]
Since $\N$ can be viewed as a subcategory of $\C$, it follows that
\[
\sum_{\bfy < \bfx \atop f \in \N(\bfy, \bfx)} V_f(V_{\bfy}) \subsetneqq V_{\bfx},
\]
that is, $\bfx \in \supp(H_0(V \downarrow_{\N}^{\C}))$. Consequently, $\gd(V \downarrow_{\N}^{\C}) \geqslant \gd(V)$, so the first equality follows.

Now we turn to the second equality. On one hand, we have
\[
\hd_1(V) \leqslant \gd(W) = \gd(W \downarrow_{\N}^{\C}) \leqslant \max\{\gd(V \downarrow_{\N}^{\C}), \, \hd_1(V \downarrow_{\N}^{\C}) \} = \pd(V \downarrow_{\N}^{\C})
\]
by considering the above two short exact sequences. By the same argument, we also have
\[
\hd_1(V \downarrow_{\N}^{\C}) \leqslant \gd(W \downarrow_{\N}^{\C}) = \gd(W) \leqslant \max\{\gd(V), \, \hd_1(V) \} = \pd(V).
\]
Consequently, one has $\pd(V) \leqslant \pd(V \downarrow_{\N}^{\C})$ and $\pd(V \downarrow_{\N}^{\C}) \leqslant \pd(V)$, and the second equality holds.
\end{proof}

The main result of this section is a combination of Theorem \ref{main result 1} and Theorem \ref{main result 2}.

\begin{theorem} \label{thm I}
Let $\| \bullet \|$ be the sup norm on $\Ob(\C)$, and $V$ a $\C$-module. Then for $s \geqslant 0$, one has
\[
\hd_s(V) \leqslant \pd(V),
\]
so the category of $\C$-modules presented in finite degrees is abelian. Moreover, if $k$ is a field and $G_{\bfx}$ is a finite group for each $\bfx \in \Ob(\C)$, then the category of finitely presented $\C$-modules is abelian.
\end{theorem}

\begin{proof}
Without loss of generality we can assume that $V$ is presented in finite degrees. The conclusion holds trivially for $s \leqslant 1$. As before, take a suitable short exact sequence $0 \to W \to P \to V \to 0$ of $\C$-modules, which induces a short exact sequence of $\N$-modules:
\[
0 \to W \downarrow_{\N}^{\C} \to P \downarrow_{\N}^{\C} \to V \downarrow_{\N}^{\C} \to 0,
\]
For $s \geqslant 2$, we have
\begin{align*}
\hd_s(V) & = \hd_{s-1}(W) \leqslant \max\{ \gd(W), \, \hd_1(W) \} & \text{ by the induction hypothesis}\\
& = \max\{ \gd(W \downarrow_{\N}^{\C}), \, \hd_1(W \downarrow_{\N}^{\C}) \} & \text{ by Lemma \ref{compare hds}}\\
& = \max\{\gd(W), \, \hd_2(V \downarrow_{\N}^{\C}) \} & \text{ by Lemma \ref{compare hds}}\\
& \leqslant \max\{\gd(W), \, \gd(V \downarrow_{\N}^{\C}), \, \hd_1(V \downarrow_{\N}^{\C}) \} & \text{ by Proposition \ref{bound hds}}\\
& = \max\{\gd(W), \, \gd(V), \, \hd_1(V)\} & \text{ by Lemma \ref{compare hds}}\\
& = \max\{\gd(V), \, \hd_1(V)\}
\end{align*}
as desired.

Now suppose that $k$ is a field and each $G_{\bfx}$ is a finite group. The argument in the proof of Lemma \ref{compare hds} tells us that the $\C$-module $V$ is finitely generated if and only if so is the $\N$-module $V \downarrow_{\N}^{\C}$. The second statement then follows from this observation as well as Lemma \ref{submodules of free module} and Remark \ref{left coherence of kN}.
\end{proof}

\section{Representations of the Young lattice}

Throughout this section let $k$ be a field, $\N$ the poset $\mathbb{N}^{\infty}$, and let $\Y$ be the Young lattice, which is precisely the full subcategory of $\N$ consisting of objects $\bfx$ such that $x_i \geqslant x_j$ whenever $i \leqslant j$. We obtain a natural inclusion functor $\iota: \Y \to \N$ and hence a restriction functor $\iota^{\ast}: \N \Mod \to \Y \Mod$.

The category $\Y$ can also be viewed as a quotient category of $\N$ by the action of the infinite symmetric group. Explicitly, let
\[
S_{\infty} = \varinjlim S_n = \{\sigma: \mathbb{N} \to \mathbb{N} \mid \sigma \text{ is a bijection and } \sigma(n) = n \text{ for } n \gg 0 \}.
\]
This group naturally acts on $\N$ as follows: $\sigma \cdot \bfx = (x_{\sigma(1)}, \, x_{\sigma(2)}, \, \ldots)$. Moreover, $\bfx \leqslant \bfy$ if and only if $\sigma \cdot \bfx \leqslant \sigma \cdot \bfy$ for any $\sigma \in S_{\infty}$. Therefore, one can obtain a quotient poset $\N / S_{\infty}$, whose objects are orbits $[\bfx]$ under this action, and $[\bfx] \leqslant [\bfy]$ if and only if there are $\bfx' \in [\bfx]$ and $\bfy' \in [\bfy]$ such that $\bfx \leqslant \bfy$. It is easy to see that $\Y$ is isomorphic to this quotient poset. Via identifying $\Y$ with this quotient poset, we obtain a quotient functor $q: \N \to \Y$ which induces a fully faithful lift functor $q^{\ast}: \Y \Mod \to \N \Mod$. Explicitly, given a $\Y$-module $W$, the corresponded $\N$-module $q^{\ast} W$ has the following structure: for a morphism $f: \bfx \to \bfy$ in $\N$, one has $(q^{\ast} W)_{\bfx} = W_{[\bfx]}$, $(q^{\ast} W)_{\bfy} = W_{[\bfy]}$, and $(q^{\ast} W)_f$ is defined to be $W_{[f]}: W_{[\bfx]} \to W_{[\bfy]}$.

Since the composite $q \iota$ is the identity functor on $\Y$, it follows that $\iota^{\ast} q^{\ast}$ is isomorphic to the identity functor on $\Y \Mod$. Furthermore, the image of $q^{\ast}$ is the full subcategory of \textit{$S_{\infty}$-symmetric $\N$-modules}. To define this notion, let $V$ be an $\N$-module and $\sigma$ an element in $S_{\infty}$. One can define another $\N$-module $\sigma V$ as follows: for a morphism $f: \bfx \to \bfy$ in $\N$, one lets $(\sigma V)_{\bfx} = V_{\sigma \cdot \bfx}$, $(\sigma V)_{\bfy} = V_{\sigma \cdot \bfy}$, and $(\sigma V)_f$ is defined to be $V_{\sigma \cdot f}: V_{\sigma \cdot \bfx} \to V_{\sigma \cdot \bfy}$. We say that $V$ is $S_{\infty}$-symmetric if $\sigma V$ is isomorphic to $V$ for every $\sigma \in S_{\infty}$.

\begin{proposition} \label{antichain}
Let $k$ be a field. Then the category of finitely generated $\Y$-modules over $k$ is abelian.
\end{proposition}

\begin{proof}
It suffices to show that every submodule $V$ of $k\Y(\bfx, -)$ is finitely generated for $\bfx \in \Ob(\Y)$. If $V$ is not finitely generated, then since the dimension of $V_{\bfy}$ is either 1 or 0 for $\bfy \in \Ob(\Y)$, it follows that $\supp(H_0(V))$ must be an infinite set. Furthermore, elements in this set are pairwise incomparable. Therefore, we get an infinite anti-chain in the poset $\Y$. This is impossible. Indeed, since $\mathbb{N}$ is an well-ordered set, it follows from Higman's lemma (see \cite{SSW}) that the set of finite sequences over $\mathbb{N}$ is quasi-well-ordered with respect to the following ordering:
\[
\bfx = (x_1, \, x_2, \, \ldots, \, x_m) \preccurlyeq \bfy = (y_1, \, y_2, \, \ldots, \, y_n)
\]
if and only if there is a strictly increasing map $\sigma: [m] \to [n]$ such that $x_i \leqslant y_{\sigma(i)}$ for each $i \in [m]$. But it is easy to see that $\bfx \preccurlyeq \bfy$ if and only if $x_i \leqslant y_i$ since entries in $\bfx$ and $\bfy$ are listed in a decreasing way; that is, $\preccurlyeq$ coincides with the usual ordering $\leqslant$ used in this paper. Consequently, $\leqslant$ is a quasi-well-order, and hence the poset $\Y$ cannot have an infinite anti-chain.
\end{proof}

\begin{remark}
In other words, the incidence algebra $k\Y$ is locally Noetherian, or the incidence algebra $k\N$ is $S_{\infty}$-symmetric Noetherian: if an $S_{\infty}$-symmetric $\N$-module is generated by its values on objects lying in only finitely many $S_{\infty}$-orbits and each such value is finite dimensional, then its $S_{\infty}$-symmetric submodules also have this property. Parallel results for $A = k[X_1, \, X_2, \, \ldots]$ were established in the literature, see for instances \cite{AH, LNNR}.
\end{remark}

Note that the sum norm on $\Ob(\N)$ induces the sum norm on $\Y$, which is compatible with the poset structure on $\Ob(\Y)$. One can also define homological degrees of $\Y$-modules with respect to this norm. For an $\N$-module, the finiteness of its zeroth and first homological degrees cannot guarantee the finiteness of the second homological degree, as shown in Example \ref{counterexample}. But for $\Y$-modules, we have a different situation.

\begin{theorem} \label{hds of repns of young lattice}
Let $k$ be a field, $\| \bullet \|$ the sum norm on $\Ob(\Y)$, and $V$ a $\Y$-module. If $\Y$ is presented in finite degrees, then $\hd_s(V)$ is finite for all $s \geqslant 1$. In particular, the category of $\Y$-modules presented in finite degrees is abelian.
\end{theorem}

\begin{proof}
We only prove the first statement since the second one follows immediately. Take a short exact sequence $0 \to W \to P \to V \to 0$ such that $P$ is a projective $\Y$-module and $H_0(P) \cong H_0(V)$ (the reader can easily check that this is possible). As in the proof of Proposition \ref{bound hds}, it suffices to show that $\hd_2(V) = \hd_1(W)$ is finite. We give a proof similar to that of Lemma \ref{submodules of free module}.

Without loss of generality suppose that $W$ is nonzero and $\gd(W) = n$. Let $S = \supp (H_0(W))$, which is a nonempty finite set, and let $W(\bfx)$ be the submodule of $W$ generated by $W_{\bfx}$. Then
\[
W = \sum_{\bfx \in S} W(\bfx).
\]
Since $W(\bfx)$ is torsion free, it is actually a direct sum of $k\Y(\bfx, -)$. Thus we obtain a surjection
\[
Q = \bigoplus_{\bfx \in S} W(\bfx) \longrightarrow \sum_{\bfx \in S} W(\bfx) = W,
\]
and hence a short exact sequence $0 \to K \to Q \to W \to 0$, from which we deduce that $\hd_1(W) \leqslant \gd(K)$.

Now we investigate $K$. Fix $\bfz \in \Ob(\Y)$ such that $K_{\bfz} \neq 0$. Then $\bfz \geqslant \bfx$ for a certain $\bfx \in S$. Define
\[
T = S \cap \{ \bfx \in \Ob(\N) \mid \bfx \leqslant \bfz \}
\]
which is again a nonempty finite set. We have
\[
K_{\bfz} = \{(W_{\bfx \to \bfz} (w_{\bfx}))_{\bfx \in T} \mid w_{\bfx} \in W_{\bfx}, \, \sum_{\bfx \in T} W_{\bfx \to \bfz} (w_{\bfx}) = 0 \},
\]
where $\bfx \to \bfz$ is the morphism from $\bfx$ to $\bfz$.

Note that $\bfz$ is an upper bound of elements in $T$. Let $\tilde{\bfz}$ be the least upper bound of elements in $T$, which always exists since $T$ is a finite set. Then $\bfz \geqslant \tilde{\bfz}$. However, if $\bfz > \tilde{\bfz}$, then
\[
W_{\bfx \to \bfz} (w_{\bfx}) = W_{\tilde{\bfz} \to \bfz} (W_{\bfx \to \tilde{\bfz}} (w_{\bfx})).
\]
Note that the map $W_{\tilde{\bfz} \to \bfz}$ is injective since $W$ as a submodule of $Q$ is torsion free. Consequently, we have
\[
\sum_{\bfx \in T} W_{\bfx \to \tilde{\bfz}} (w_{\bfx}) = 0,
\]
so $(W_{\bfx \to \tilde{\bfz}} (w_{\bfx}))_{\bfx \in T}$ is contained $K_{\tilde{\bfz}}$. Furthermore, the map $K_{\tilde{\bfz} \to \bfz}$ sends $(W_{\bfx \to \tilde{\bfz}} (w_{\bfx}))_{\bfx \in T}$ to $(W_{\bfx \to \bfz} (w_{\bfx}))_{\bfx \in T}$. Consequently, the map $K_{\tilde{\bfz} \to \bfz}$ is surjective, and hence $\bfz$ is not contained in $\supp(H_0(K))$.

We have shown that $\bfz \in \supp(H_0(K))$ only if there is a nonempty finite set $T$ of $S$ such that $\bfz$ is precisely the least upper bound of elements in $T$. Since $S$ is a finite set, there are only finitely many such least upper bounds. Consequently, $\supp(H_0(K))$ is a finite set, so $\gd(K)$ is finite as well. This finishes the proof.
\end{proof}

\begin{remark}
By comparing the proof of this theorem and Example \ref{counterexample}, the reason why the above conclusion is not valid for $\N$-modules $V$ becomes transparent: in that case $\supp(H_0(V))$ in general is an infinite set, so it has infinitely many finite subsets and correspondingly one can get infinitely many least upper bounds.
\end{remark}

\begin{remark}
Note that for each fixed $n \in \mathbb{N}$, there are only finitely many objects $\bfx \in \Ob(\Y)$ such that the sum norm $\| \bfx \| = n$. Therefore, a $\Y$-module $V$ is presented in finite degrees if and only if it is presented in finitely many objects, an intrinsic property independent of norms on $\Ob(\Y)$.
\end{remark}

Finally we describe a recursive upper bound for homological degrees of $\Y$-modules.

\begin{corollary}
Let $k$ be a field, $\| \bullet \|$ the sum norm on $\Ob(\Y)$, and $V$ a $\Y$-module. If $s \geqslant 2$ and $\hd_{s-1}(V) > 0$, then one has
\[
\hd_s(V) \leqslant \hd_{s-1}(V) (\ln(\hd_{s-1}(V)) + 1).
\]
\end{corollary}

Note that $\hd_{s-1}(V) = -1$ if and only if $H_{s-1}(V) = 0$, and in this case $\hd_s(V) = -1$. If $\hd_{s-1}(V) = 0$, then there is a projective resolution $\ldots \to P^1 \to V \to 0$ such that the image $V^{s-1}$ of $P^s \to P^{s-1}$ is a torsion free $\Y$-module generated by the value of $V^{s-1}$ on the object $\boldsymbol{0}$. But it follows that $V^{s-1}$ is projective, so $\hd_s(V) = -1$ as well.

\begin{proof}
It suffices to prove the inequality for $s = 2$. Let $W$ and $K$ be as in the proof of Theorem \ref{hds of repns of young lattice}. Since we have assumed that $H_0(P) \cong H_0(V)$, it follows that $\hd_1(V) = \gd(W)$. Without loss of generality we can assume this number to be $n$. Let $T$ be the subset of $\Ob(\Y)$ consisting of objects $\bfy$ with $\| \bfy \| = n$, and let $S = \supp(H_0(W))$. For every $\bfx \in S$, since $\| \bfx \| \leqslant n$, we can find a certain $\bfy \in T$ such that $\bfx \leqslant \bfy$. Consequently, the least upper bound $\bigvee_{\bfx \in S} \bfx$ of $S$ is less than or equal to $\bigvee_{\bfx \in T} \bfx$.

It is not hard to see from the proof of Theorem \ref{hds of repns of young lattice} that $\gd(K)$ is bounded by $\| \bigvee_{\bfx \in S} \bfx \|$, so
\[
\hd_2(V) \leqslant \gd(K) \leqslant \| \bigvee_{\bfx \in S} \bfx \| \leqslant \| \bigvee_{\bfx \in T} \bfx \|.
\]
But $T$ consists of all partitions of $[n]$ which have a common upper bound $(\lambda_1, \ldots, \lambda_n)$ with $\lambda_i = \lfloor n/i \rfloor \leqslant n/i$. Consequently, we have
\[
\hd_2(V) \leqslant \| \bigvee_{\bfx \in T} \bfx \| \leqslant \sum_{i=1}^n \frac{n}{i} = n \sum_{i=1}^n \frac{1}{n} \leqslant n(\ln n + 1).
\]
\end{proof}

\section{A sheaf theoretic interpretation}

In this section we use sheaf theory over ringed sites to give an interpretation of results established in previous sections from another viewpoint. Let $\mathcal{C}$ be a small category and $k$ a commutative ring. Recall that a presheaf of $k$-modules over $\mathcal{C}^{\op}$ is precisely a $\mathcal{C}$-module. Suppose that $\mathcal{C}$ satisfies the left Ore condition, so one can impose the atomic Grothendieck topology $J_{at}$ on $\mathcal{C}^{\op}$. Let $\boldsymbol{\mathcal{C}}^{\op}$ be the site $(\mathcal{C}^{\op}, \, J_{at})$, and let $\underline{k}$ be the constant structure sheaf. For details about atomic topology and sheaves over ringed sites, please refer to \cite{DLLX}.

The following result is taken from \cite{DLLX}.

\begin{theorem} \label{sheaves}
A $\mathcal{C}$-module $V$ is a sheaf over the ringed site $(\boldsymbol{\mathcal{C}}^{\op}, \underline{k})$ if and only if
\[
\Hom_{\mathcal{C} \Mod} (T, V) = 0 = \Ext_{\mathcal{C} \Mod}^1 (T, V)
\]
for any torsion $\mathcal{C}$-module $T$. The category $\Sh(\boldsymbol{\mathcal{C}}^{\op}, \, \underline{k})$ of sheaves of $k$-modules is equivalent to the Serre quotient $\mathcal{C} \Mod / \mathcal{C} \Mod^{\tor}$, where $\mathcal{C} \Mod^{\tor}$ is the full subcategory of $\mathcal{C} \Mod$ consisting of torsion $\C$-modules.
\end{theorem}

When $\mathcal{C}$ is a poset satisfying the left Ore condition, we have a very trivial description for $\Sh(\boldsymbol{\mathcal{C}}^{\op}, \, \underline{k})$.

\begin{corollary} \label{classify sheaves}
Every sheaf over $(\boldsymbol{\mathcal{C}}^{\op}, \, \underline{k})$ is a constant sheaf, and hence $\Sh(\boldsymbol{\mathcal{C}}^{\op}, \, \underline{k}) \simeq k \Mod$.
\end{corollary}

\begin{proof}
An explicit equivalence is given as follows. Firstly, since $\mathcal{C}^{\op}$ is equipped with the atomic topology, $\mathcal{C}$ is a directed set. Therefore, the direct limit functor
\[
\varinjlim: \mathcal{C} \Mod \to k \Mod
\]
is exact. It has a right adjoint, the diagonal functor $\Delta: k \Mod \to \mathcal{C} \Mod$ sending each $k$-module to the corresponded constant $\mathcal{C}$-module. Furthermore, a $\mathcal{C}$-module $V$ is sent to 0 by $\varinjlim$ if and only if $V$ is torsion. Therefore, by Theorem \ref{sheaves} and \cite[Proposition III.2.5]{Gab}, we have
\[
\Sh(\boldsymbol{\mathcal{C}}^{\op}, \, \underline{k}) \simeq \mathcal{C} \Mod / \mathcal{C} \Mod^{\tor} \simeq k \Mod,
\]
and moreover each sheaf over $(\boldsymbol{\mathcal{C}}^{\op}, \, \underline{k})$ is an image of $\Delta$, so must be a constant sheaf.
\end{proof}

\subsection{$\BS^{\infty}$ as the sheafification functor}

Throughout this subsection let $\N$ be the poset $\mathbb{N}^{\infty}$. Recall that the inclusion functor $\textsl{inc}: \Sh(\boldsymbol{\N}^{\op}, \, \underline{k}) \to \N \Mod$ has a left adjoint $\sharp$, the \textit{sheafification functor}, where $\boldsymbol{\N}^{\op} = (\N^{\op}, J_{at})$. It follows from the proof of Corollary \ref{classify sheaves} that $\sharp$ is isomorphic to the following composite of functors
\[
\xymatrix{
\N \Mod \ar[r]^-{\varinjlim} & k \Mod \ar[r]^-{\Delta} & \Sh(\boldsymbol{\N}^{\op}, \, \underline{k}).
}
\]
Furthremore, by Proposition \ref{Sigma infinity} and Corollary \ref{classify sheaves}, the shift functor $\BS^{\infty}$ can be viewed as a functor from $\N \Mod$ to $\Sh(\boldsymbol{\N}^{\op}, \, \underline{k})$. Surprisingly, these two functors coincide up to isomorphism.

\begin{proposition} \label{sheafification}
The functor $\BS^{\infty}: \N \Mod \to \Sh(\boldsymbol{\N}^{\op}, \, \underline{k})$ is isomorphic to the sheafification functor.
\end{proposition}

\begin{proof}
Take $V$ in $\N \Mod$ and $W$ in $\Sh(\boldsymbol{\N}^{\op}, \, \underline{k})$. We want to construct a natural isomorphism
\[
\Hom_{\N \Mod} (\BS^{\infty} V, \, W) \cong \Hom_{\N \Mod} (V, \, W);
\]
that is, $\BS^{\infty}$ is the left adjoint of $\textsl{inc}$, where on the right side we write $W$ for $\textsl{inc} \, W$ for brevity.

Applying $\BS^{\infty}$ to the short exact sequence $0 \to V_T \to V \to V_F \to 0$ we get $\BS^{\infty} V \cong \BS^{\infty} V_F$ by Proposition \ref{Sigma infinity}. Therefore, without loss of generality we can assume that $V$ is torsion free, so by Proposition \ref{Sigma infinity} we have another short exact sequence
\[
0 \to V \to \BS^{\infty} V \to U \to 0,
\]
which induces a long exact sequence
\[
0 \to \Hom_{\N \Mod} (U, \, W) \to \Hom_{\N \Mod} (\BS^{\infty} V, \, W) \to \Hom_{\N \Mod} (V, \, W) \to \Ext_{\N \Mod}^1 (U, \, W) \to \ldots.
\]
By Remark \ref{cokernel is torsion}, $U$ is a torsion $\N$-module. Since $W$ is a sheaf over $(\boldsymbol{\N}^{\op}, \, \underline{k})$, Theorem \ref{sheaves} asserts that the first term and the last term vanish. Consequently, we have
\[
\Hom_{\N \Mod} (\BS^{\infty} V, \, W) \cong \Hom_{\N \Mod} (V, \, W).
\]
Furthermore, all constructions to establish this isomorphism are natural, so the conclusion holds.
\end{proof}

\begin{remark}
We give a more illustrative explanation that $\BS^{\infty}$ indeed coincides with the sheafification functor. Since $\BS^{\infty} V$ is the colimit of the following sequence of $\N$-module homomorphisms
\[
V \to \BS V \to \BS^2 V \to \BS^3 V \to \ldots,
\]
$\Hom_{\N \Mod} (\BS^{\infty} V, \, W)$ is isomorphic to the set of sequences $(\phi^0, \, \phi^1, \, \phi^2, \, \ldots)$ of $\N$-module homomorphisms such that the following diagram commutes:
\[
\xymatrix{
V \ar[r] \ar[d]^-{\phi^0} & \BS V \ar[r] \ar[d]^-{\phi^1} & \BS^2 V \ar[r] \ar[d]^{\phi^2} & \ldots\\
W \ar[r]^-{\text{id}} & W \ar[r]^-{\text{id}} & W \ar[r]^-{\text{id}} & \ldots.
}
\]
Clearly, this sequence determines an element in $\Hom_{\N \Mod} (V, \, W)$. The above proposition asserts that the converse is also true: every element in $\Hom_{\N \Mod}(V, \, W)$ induces such a sequence if $W$ is a sheaf over $(\boldsymbol{\N}^{\op}, \underline{k})$, or equivalently a constant $\N$-module.

To see the reason, let us construct $\phi^1: \BS V \to W$ from $\phi^0: V \to W$. For $\bfx \in \Ob(\N)$, we need to define $\phi^1_{\bfx}: (\BS V)_{\bfx} \to W_{\bfx}$. Since $(\BS V)_{\bfx}$ is the colimit of the following sequence
\[
V_{\bfx} \longrightarrow (\Sigma_1 V)_{\bfx} = V_{\bfx + \bfo_1} \longrightarrow (\Sigma_{[2]} V)_{\bfx} = V_{\bfx + \bfo_1 + \bfo_2} \longrightarrow \ldots,
\]
it suffices to define vertical maps to obtain a commutative diagram of $k$-modules
\[
\xymatrix{
V_{\bfx} \ar[r] \ar[d] &  V_{\bfx + \bfo_1} \ar[r] \ar[d] & V_{\bfx + \bfo_1 + \bfo_2} \ar[r] \ar[d] & \ldots\\
W_{\bfx} \ar[r] & W_{\bfx} \ar[r] & W_{\bfx} \ar[r] & \ldots.
}
\]
Since $\phi^0$ gives a commutative diagram
\[
\xymatrix{
V_{\bfx} \ar[r] \ar[d]^-{\phi^0_{\bfx}} &  V_{\bfx + \bfo_1} \ar[r] \ar[d]^-{\phi^0_{\bfx + \bfo_1}} & V_{\bfx + \bfo_1 + \bfo_2} \ar[r] \ar[d]^-{\phi^0_{\bfx + \bfo_1 + \bfo_2}} & \ldots\\
W_{\bfx} \ar[r] & W_{\bfx + \bfo_1} \ar[r] & W_{\bfx + \bfo_1 + \bfo_2} \ar[r] & \ldots,
}
\]
and $W$ is a constant $\N$-module, we can successfully define vertical maps in the first diagram.
\end{remark}

\begin{remark}
Given a torsion free $\N$-module $V$ (in other words, a separated presheaf over $(\boldsymbol{\N}^{\op}, \underline{k})$), a standard fact of sheaf theory tells us that the natural map $V \to V^{\sharp} \cong \BS^{\infty} V$ is injective, which is precisely the second statement of Proposition \ref{Sigma infinity}.
\end{remark}

\subsection{Sheaves over the ringed site $(\mathscr{Z}, \, J_{at}, \, \underline{k})$}

In this subsection let $\mathscr{Z}$ be the category defined in Example \ref{orbit cat of Z}. Since $\mathscr{Z}$ satisfies the right Ore condition, we can impose the atomic Grothendieck topology on $\mathscr{Z}$, and consider sheaves over the ringed site $(\mathscr{Z}, \, J_{at}, \, \underline{k})$. Throughout this subsection let $k$ be an algebraically closed field with characteristic 0.

Recall that every object in $\mathscr{Z}$ is the coset $\bfn = \mathbb{Z}/n\mathbb{Z}$, and morphisms $f: \bfm \to \bfn$ are surjective $(\mathbb{Z}, +)$-equivariant maps. In particular, endomorphisms from $\bfn$ to itself form the cyclic group $G_n = \langle g \mid g^n = 1 \rangle$. If $V$ is a presheaf of $k$-modules over $\mathscr{Z}$, or equivalently, an $\mathscr{Z}^{\op}$-module, since the canonical morphism $f: \bfm \to \bfn$ in $\mathscr{Z}(\bfm, \bfn)$ sending $\bar{1} \in \bfm$ to $\bar{1} \in \bfn$ induces an embedding functor $kG_n \Mod \to kG_m \Mod$, one can regard $V_f: V_n \to V_m$ as a $kG_m$-module homomorphism. In particular, let $\xi$ be a $n$-th primitive root of 1, then the corresponded irreducible $kG_n$-module $L_{\xi}$ can be regarded as an irreducible $kG_m$-module while $n$ divides $m$. Denote by $\underline{L}_{\xi}$ the $\mathscr{Z}^{\op}$-submodule (actually, an irreducible direct summand) of $k\mathscr{Z}(-, \bfn)$ generated by $L_{\xi}$ on $\bfn$. Explicitly, its value on $\bfm$ is $L_{\xi}$ if $n$ divides $m$ and $0$ otherwise.

The following result classifies all indecomposable injective objects in $\Sh(\mathscr{Z}, \, J_{at}, \, \underline{k})$, which are precisely indecomposable torsion free $\mathscr{Z}^{\op}$-modules by \cite[Corollary 3.9]{DLLX}.

\begin{theorem}
Let $k$ be an algebraically closed field with characteristic 0. Then $\xi \mapsto \underline{L}_{\xi}$ gives a bijective correspondence between the set of all primitive roots of 1 and the set of isomorphism classes of indecomposable injective torsion free $\mathscr{Z}^{\op}$-modules.
\end{theorem}

\begin{proof}
Suppose that $\xi$ is a primitive $n$-th root of 1. Clearly, $\underline{L}_{\xi}$ is torsion free and indecomposable. We show that it is injective; that is, the functor $\Hom_{\mathscr{Z}^{\op} \Mod} (-, \, \underline{L}_{\xi})$ is exact. Given an $\mathscr{Z}^{\op}$-module $V$, let $V'$ be the submodule satisfying the following condition: for every $\bfm \in \Ob(\mathscr{Z}^{\op})$, $V'_{\bfm}$ is a direct sum of all copies of $L_{\xi}$ appearing in $V_{\bfm}$. This submodule is well defined by Schur's lemma. Moreover, by Schur's lemma again, one has a decomposition $V = V' \oplus V''$. Clearly, one has
\[
\Hom_{\mathscr{Z}^{\op} \Mod} (V, \, \underline{L}_{\xi}) \cong \Hom_{\mathscr{Z}^{\op} \Mod} (V', \, \underline{L}_{\xi}).
\]
Since for each object $\bfm$ such that $n$ does not divide $m$, $L_{\xi}$ cannot appear in $V_{\bfm}$, we can regard $V'$ and $\underline{L}_{\xi}$ as an $\mathscr{Z}_{\geqslant \bfn}^{\op}$-module, where $\mathscr{Z}_{\geqslant \bfn}^{\op}$ is the full subcategory of $\mathscr{Z}^{\op}$ consisting of all objects $\bfm$ such that $n$ divides $m$. Clearly, one has
\[
\Hom_{\mathscr{Z}^{\op} \Mod} (V', \, \underline{L}_{\xi}) \cong \Hom_{\mathscr{Z}_{\geqslant \bfn}^{\op} \Mod} (V', \, \underline{L}_{\xi}).
\]

Note that $L_{\xi}$ is viewed as a $kG_m$-module via the natural quotient group homomorphism $G_m \to G_n$ when $n$ divides $m$. Let $\mathscr{E}$ be the following category:
\begin{itemize}
\item $\Ob(\mathscr{E}) = \Ob(\mathscr{Z}^{\op}_{\geqslant \bfn})$;

\item for $\bfm \in \Ob(\mathscr{E})$, set $H_{\bfm} = \mathscr{E}(\bfm, \, \bfm)$ be the unique quotient group of $G_m$ which is isomorphic to $G_n$;

\item for $\bfm, \, \bfx \in \Ob(\mathscr{E})$, $\mathscr{E} (\bfm, \, \bfx)$ are isomorphisms between $H_{\bfm}$ and $H_{\bfx}$.
\end{itemize}
Then there is a natural quotient functor $q: \mathscr{Z}_{\geqslant \bfn}^{\op} \to \mathscr{E}$ which induces a fully faithful functor $q^{\ast}: \mathscr{E} \Mod \to \mathscr{Z}_{\geqslant \bfn}^{\op} \Mod$, and one can view $V'$ and $\underline{L}_{\xi}$ as $\mathscr{E}$-modules. Furthermore, one has
\[
\Hom_{\mathscr{Z}_{\geqslant \bfn}^{\op} \Mod} (V', \, \underline{L}_{\xi}) \cong \Hom_{\mathscr{E} \Mod} (V', \, \underline{L}_{\xi}).
\]

We observe that each step in the following procedure
\[
V \in \mathscr{Z}^{\op} \Mod \mapsto V' \in \mathscr{Z}^{\op} \Mod \mapsto V' \in \mathscr{Z}_{\geqslant \bfn}^{\op} \Mod \mapsto V' \in \mathscr{E} \Mod
\]
is exact. Therefore, to show that $\underline{L}_{\xi}$ is injective in $\mathscr{Z}^{\op} \Mod$, it suffices to show that $\underline{L}_{\xi}$ is injective in $\mathscr{E} \Mod$. However, using the same argument as in the proof of Corollary \ref{classify sheaves}, one can show that
\[
\Sh(\mathcal{E}^{\op}, \, J_{at}, \, \underline{k}) \simeq kG_n \Mod.
\]
In particular, $\underline{L}_{\xi}$ is an injective sheaf since it corresponds to the $kG_n$-module $L_{\xi}$ which is injective in $kG_n \Mod$. Consequently, it is a torsion free injective $\mathcal{E}$-module by \cite[Corollary 3.9]{DLLX}.

We have show that every $\underline{L}_{\xi}$ is a torsion free injective $\mathscr{Z}^{\op}$-module. Now we prove the converse statement; that is, every indecomposable torsion free injective $\mathscr{Z}^{\op}$-module $I$ is isomorphic to $\underline{L}_{\xi}$ for a certain primitive root $\xi$ of 1. Take a minimal object $\bfn$ in $\mathscr{Z}^{\op}$ such that $I_{\bfn} \neq 0$, and let $\underline{L}$ be the submodule of $I$ generated by a simple summand $L$ of $I_{\bfn}$. For any $f: \bfx \to \bfy$ in $\mathscr{Z}^{\op}$, applying the snake lemma to the commutative diagram of short exact sequences
\[
\xymatrix{
0 \ar[r] & \underline{L}_{\bfx} \ar[r] \ar[d]^-{\underline{L}(f)} & I_{\bfx} \ar[r] \ar[d]^-{I(f)} & (I/\underline{L})_{\bfx} \ar[r] \ar[d]^-{(I/\underline{L})(f)} & 0\\
0 \ar[r] & \underline{L}_{\bfy} \ar[r] & I_{\bfy} \ar[r] & (I/\underline{L})_{\bfy} \ar[r] & 0
}
\]
we know that the third vertical map is injective since the first one is an isomorphism and the second one is injective. Consequently, $I/\underline{L}$ is a torsion free $\mathscr{Z}^{\op}$-module, so $\underline{L}$ is a sheaf over the ringed site $(\mathscr{Z}^{\op}, \, J_{at}, \, \underline{k})$ by \cite[Proposition 3.7 and Theorem 3.8]{DLLX}.

Note that $L$ corresponds to a unique $n$-th root $\xi$ of 1 which is an $m$-th primitive root for a unique $m$ dividing $n$. Letting $\underline{L}_{\xi}$ be the injective torsion free module corresponded to $\xi$ constructed as before, we obtain a short exact sequence $0 \to \underline{L} \to \underline{L}_{\xi} \to \underline{L}_{\xi}/\underline{L} \to 0$. By \cite[Proposition 3.7 and Theorem 3.8]{DLLX}, the third term must be torsion free. However, since its value on $\bfn$ is 0, this happens if and only if $m = n$; that is, $\underline{L} \cong \underline{L}_{\xi}$. Since $I$ is an indecomposable $\mathscr{Z}^{\op}$-module containing $\underline{L}$ as a submodule, we must have $I \cong \underline{L}_{\xi}$.
\end{proof}

The above proof actually implies that indecomposable injective objects in $\Sh(\mathscr{Z}, \, J_{at}, \, \underline{k})$ coincide with irreducible objects. Consequently, $\Sh(\mathscr{Z}, \, J_{at}, \, \underline{k})$, is a semisimple category, as asserted in the following corollary.

\begin{corollary}
Let $k$ be an algebraically closed field with characteristic 0. Then every object in $\Sh(\mathscr{Z}, \, J_{at}, \, \underline{k})$ is isomorphic to a direct sum of irreducible sheaves of the form $\underline{L}_{\xi}$. Consequently,
\[
\Sh(\mathscr{Z}, \, J_{at}, \, \underline{k}) \simeq \bigtimes_{\xi} k \Mod,
\]
the cartesian product of categories $k \Mod$ indexed by all primitive roots of 1.
\end{corollary}

\begin{proof}
Let $V$ be an object in $\Sh(\mathscr{Z}, \, J_{at}, \, \underline{k})$. We can obtain a short exact sequence
\[
0 \to W \to P = k\mathscr{Z}^{\op} \otimes_{k\mathscr{Z}_0} H_0(V) \to V \to 0
\]
in $\mathscr{Z}^{\op} \Mod$ where $\mathscr{Z}_0$ is the subcategory of $\mathscr{Z}^{\op}$ consisting of all objects and all invertible morphisms in $\mathscr{Z}^{\op}$. We claim that $P$ is a direct sum of torsion free injective $\mathscr{Z}^{\op}$-modules. If this is true, then the Grothendieck category $\Sh(\mathscr{Z}, \, J_{at}, \, \underline{k})$ has a set of irreducible generators
\[
\{ \underline{L}_{\xi} \mid \xi \text{ is an $n$-th primitive root for a certain } n \in \mathbb{N}_+ \}.
\]
The conclusions then follow from this fact. Indeed, in this case $P$ is a semismple object in $\Sh(\mathscr{Z}, \, J_{at}, \, \underline{k})$. Since a simple summand in $P$ becomes either 0 or a simple subobject of $V$ via the quotient map $P \to V$, and $V$ is the sum of these simple subobjects, it follows that $V$ is semisimple by \cite[Theorem 2.3]{PS}.

Now we prove the claim. Note that every indecomposable direct summand $\underline{L}$ of $P$ is generated by a certain $L_{\xi}$ in $kG_n \Mod$ with $\xi$ an $n$-th root of 1, so $L_{\xi}$ can be embedded into $\underline{L}_{\xi}$ constructed before, and hence we obtain a commutative diagram
\[
\xymatrix{
0 \ar[r] & W \ar[r] \ar@{=}[d] & P \ar[r] \ar[d]^-{\beta}         & V \ar[r] \ar[d]   & 0\\
0 \ar[r] & W \ar[r]            & \tilde{P} \ar[d] \ar[r] & \tilde{V} \ar[r] \ar[d] & 0\\
         &                     & T \ar@{=}[r]            & T
}
\]
where $\tilde{P}$ is a direct sum of those $\underline{L}_{\xi}$ and $\beta$ is the map induced by these embeddings. Clearly, $T$ is a torsion $\mathscr{Z}^{\op}$-module (including the special case that $T = 0$). Since $V$ is a sheaf, by Theorem \ref{sheaves}, the third column must split. Applying the functor $H_0$ we obtain a commutative diagram
\[
\xymatrix{
 & H_0(P) \ar[r]^-{\alpha} \ar@{=}[d] & H_0(\tilde{P}) \ar[r] \ar[d] & H_0(T) \ar[r] \ar@{=}[d] & 0\\
0 \ar[r] & H_0(V) \ar[r] & H_0(V) \oplus H_0(T) \ar[r] & H_0(T) \ar[r] & 0.
}
\]
It turns out that the middle vertical arrow is also an isomorphism, and hence the map $\alpha$ is a split monomorphism. It induces a split monomorphism
\[
P = k\mathscr{Z}^{\op} \otimes_{k\mathscr{Z}_0} H_0(V) \to \tilde{P} = k\mathscr{Z}^{\op} \otimes_{k\mathscr{Z}_0} H_0(\tilde{P}).
\]
But by our construction of $\tilde{P}$, this happens if and only if $P \cong \tilde{P}$, so the claim is valid.
\end{proof}

\begin{remark}
Note that $k_{\xi}$ is isomorphic to the endomorphism ring of $\underline{L}_{\xi}$, so $\underline{L}_{\xi}$ is quipped with a $(k\mathscr{Z}^{\op}, \, k_{\xi})$-bimodule structure. With this observation, the equivalence is given by the functors:
\[
\xymatrix{
\bigoplus_{\xi} \Hom_{\mathscr{Z}^{\op} \Mod} (\underline{L}_{\xi}, \, -): \Sh(\mathscr{Z}, \, J_{at}, \, \underline{k}) \ar@<.5px>@^{->}[r] & \bigtimes_{\xi} (k_{\xi} \Mod) \simeq (\bigoplus_{\xi} k_{\xi}) \Mod: \bigoplus_{\xi} (\underline{L}_{\xi} \otimes_{k_{\xi}} -) \ar@<.5ex>@^{->}[l]
}
\]
which both are exact and form an adjunction pair. Denote by $\mu$ the hom functor and $\nu$ the tensor functor. Clearly, $\nu \mu$ is isomorphic to the identity functor. Furthermore, since those $\underline{L}_{\xi}$'s form a set of generators of $\Sh(\mathscr{Z}, \, J_{at}, \, \underline{k})$, the kernel of $\mu$ is zero. Therefore, by \cite[Proposition III.2.5]{Gab}, one obtains the desired equivalence.
\end{remark}

\end{document}